\documentclass[12pt,reqno]{amsart}

\usepackage{XCharter}

\usepackage{amsmath, amssymb, amsthm}
\usepackage{mathtools}        
\usepackage{mathrsfs}         
\usepackage{stmaryrd}         

\usepackage{tikz}
\usetikzlibrary{cd, decorations.pathmorphing, arrows, positioning}

\usepackage[top=1in, bottom=1in, left=2.5cm, right=2.5cm]{geometry}
\usepackage{parskip}

\usepackage{microtype}

\usepackage[shortlabels]{enumitem}  

\usepackage{appendix}
\usepackage{verbatim}
\usepackage{comment}
\usepackage{graphicx}
\usepackage{subfiles}
\usepackage{chngcntr}

\usepackage[colorlinks, pagebackref]{hyperref}
\hypersetup{
  colorlinks = true,
  citecolor  = red,
  linkcolor  = blue,
  urlcolor   = cyan
}

\theoremstyle{definition}
\newtheorem{para}{}[section]

\theoremstyle{plain}
\newtheorem{thm}[para]{Theorem}
\newtheorem{lem}[para]{Lemma}
\newtheorem{cor}[para]{Corollary}
\newtheorem{prop}[para]{Proposition}

\newtheorem*{mainthm}{Main Theorem}

\theoremstyle{definition}
\newtheorem{defn}[para]{Definition}

\newtheorem{example}[para]{Example}

\theoremstyle{remark}
\newtheorem{rem}[para]{Remark}

\newcommand{\bbA}{\mathbb{A}}  
\newcommand{\bbC}{\mathbb{C}}  
  \newcommand{\bbF}{\mathbb{F}}
\newcommand{\bbG}{\mathbb{G}}

  \newcommand{\bbP}{\mathbb{P}}
\newcommand{\bbQ}{\mathbb{Q}}

  \newcommand{\bbZ}{\mathbb{Z}}

  \newcommand{\bH}{\mathbf{H}}

  \newcommand{\cD}{\mathcal{D}}
  \newcommand{\cF}{\mathcal{F}}
\newcommand{\cG}{\mathcal{G}}  \newcommand{\cH}{\mathcal{H}}

\newcommand{\cM}{\mathcal{M}}  
\newcommand{\cO}{\mathcal{O}}  
\newcommand{\cQ}{\mathcal{Q}}

\newcommand{\rG}{\mathrm{G}}  \newcommand{\rH}{\mathrm{H}}

\newcommand{\ol}[1]{\overline{#1}}

\newcommand{\Zar}{\mathrm{Zar}}
\newcommand{\et}{\mathrm{\acute{e}t}}
\newcommand{\proet}{\mathrm{pro\text{-}\acute{e}t}}

\newcommand{\Ql}{\bbQ_{\ell}}

\DeclareMathOperator{\coker}{coker}

\DeclareMathOperator{\Div}{Div}

\DeclareMathOperator{\Gal}{Gal}

\DeclareMathOperator{\NS}{NS}
\DeclareMathOperator{\Pic}{Pic}

\DeclareMathOperator{\Spec}{Spec}

\DeclareMathOperator{\cosk}{cosk} 
\newcommand{\ZL}{\mathrm{ZL}}   
\newcommand{\sdh}{\mathrm{sdh}} 

\title{N\'eron--Severi groups of proper schemes over finite fields}
\author{K. V. Shuddhodan}
\address{Department of Mathematics, University of Notre Dame, Notre Dame, IN 46556, USA}
\email{skadattu@nd.edu}

\author{V. Srinivas}
\address{Department of Mathematics, University at Buffalo - SUNY, Buffalo, NY 14260, USA}
\email{vs74@buffalo.edu}
\date{\today}

\begin{document}

\begin{abstract}
Let \(X\) be a proper reduced scheme over a finite field \(k\), let \(\ell\) be a prime different from
\(\operatorname{char} k\), and write \(\ol X=X\times_{k}\ol k\) for its base change to an algebraic closure
\(\ol k\) of \(k\). Call a class in \(\rH^{2}_{\et}(\ol X,\bbZ_{\ell}(1))\) \emph{Zariski-locally trivial}
if it vanishes on a Zariski-open cover of \(\ol X\). We prove that the first Chern class map
identifies \(\NS(\ol X)\otimes\bbZ_{\ell}\) with the group of Zariski-locally trivial classes whose image
in \(\rH^{2}_{\et}(\ol X,\Ql(1))\) has weight zero. This is the finite-field analogue of a theorem of
Barbieri-Viale--Rosenschon--Srinivas for proper seminormal complex varieties. In the finite-field setting
neither seminormality nor irreducibility is needed.
\end{abstract}

\maketitle
\tableofcontents
\newpage

\section{Introduction}
\label{sec:introduction}

Let \(X\) be a smooth projective complex variety. The first Chern class embeds the N\'eron--Severi group
\(\NS(X)=\Pic(X)/\Pic^0(X)\) into \(\rH^2(X,\bbZ)\) and the resulting subgroup admits two (equivalent) 
descriptions:

\begin{enumerate}[(a)]

\item  By the Lefschetz theorem on \((1,1)\)-classes it is \(F^1\cap\rH^2(X,\bbZ)\), the integral classes in the first level of the Hodge filtration.

\item  It is also the subgroup \(\rH^2_{\ZL}(X,\bbZ)\subseteq\rH^2(X,\bbZ)\) of
  \emph{Zariski-locally trivial} classes, those vanishing on some Zariski-open cover.

\end{enumerate}

 For singular \(X\) both descriptions fail. One has only the inclusions \(\NS(X)\subseteq F^1\cap\rH^2(X,\bbZ)\) and
\(\NS(X)\subseteq\rH^2_{\ZL}(X,\bbZ)\), and each can be strict \cite{BarbieriVialeSrinivas94}. However, in many
situations the intersection of the two subgroups does recover the N\'eron--Severi group, that is
\(\NS(X)=F^1\cap\rH^2_{\ZL}(X,\bbZ)\). This was proved for normal projective varieties by Biswas--Srinivas \cite{BiswasSrinivas00} and for proper seminormal varieties
by Barbieri-Viale, Rosenschon and Srinivas \cite{BRS09}. Over \(\bbC\) seminormality cannot be dropped \cite[Section 5]{BiswasSrinivas00}.

This article proves the corresponding statement over a finite field. Let \(X\) be a proper reduced scheme
over \(k=\bbF_q\), let \(\ell\) be a prime different from \(p=\operatorname{char}k\), and put
\(\ol X=X\times_k\ol k\). Call a class in \(\rH^2_{\et}(\ol X,\bbZ_{\ell}(1))\) \emph{Zariski-locally
trivial} if it restricts to zero on some Zariski-open cover of \(\ol X\), and write
\[
  \rH^2_{\ZL}(\ol X,\bbZ_{\ell}(1))\ \subseteq\ \rH^2_{\et}(\ol X,\bbZ_{\ell}(1))
\]
for the subgroup of such classes. The first Chern class \(c_1\colon\Pic(\ol X)\to\rH^2_{\et}(\ol X,\bbZ_{\ell}(1))\)
lands in this subgroup, and the main result characterises the N\'eron--Severi group by a single condition
on weights (see Corollary~\ref{cor:integral-form}).

\begin{mainthm}
With notation as above, the first Chern class induces an isomorphism
\begin{equation}\label{eq:NS-integral}
  \NS(\ol X)\otimes\bbZ_{\ell}\ \xrightarrow{\ \sim\ }\
  \bigl\{\alpha\in\rH^2_{\ZL}(\ol X,\bbZ_{\ell}(1)) : \alpha\otimes\Ql\in\rH^2_{\et}(\ol X,\Ql(1))^{w=0}\bigr\}
\end{equation}
where \(\rH^2_{\et}(\ol X,\Ql(1))^{w=0}\) is the part on which the geometric Frobenius acts with eigenvalues
of absolute value \(1\).
\end{mainthm}

Over the complex numbers the results of \cite{BiswasSrinivas00,BRS09} impose the Hodge-filtration
condition \(F^1\). The weight-zero condition plays that role here. Those results describe \(\NS\) for a
normal projective variety and for a proper seminormal variety. The theorem above holds for \emph{all}
proper reduced \(k\)-schemes.

\emph{Strategy of the proof.} Write \(\cH^1\) for the Zariski sheaf \(\cH^1_{\ol X}(\Ql(1))\) on \(\ol X\)
associated to the presheaf \(U\mapsto\rH^1_{\et}(U,\Ql(1))\). The scheme and the coefficients \(\Ql(1)\)
are suppressed from \(\cH^1\) when no confusion can arise. The rational invariant of interest is the
Zariski cohomology \(\rH^1_{\Zar}(\ol X,\cH^1)\).
In Theorem~\ref{thm:main-result} we prove the rational isomorphism
\begin{equation}\label{eq:NS-rational}
  \NS(\ol X)\otimes\Ql\ \xrightarrow{\ \sim\ }\ \rH^1_{\Zar}(\ol X, \cH^1)^{w=0},
\end{equation}
induced by the cycle class and from this we deduce the integral statement~\eqref{eq:NS-integral} of the introduction
(Corollary~\ref{cor:integral-form}).

The isomorphism~\eqref{eq:NS-rational} is proved by first establishing it for \(\ol X\) smooth and proper,
and then descent via a smooth proper hypercover. This descent argument is the crux of the matter, as in
\cite{BiswasSrinivas00,BRS09}. Over a smooth proper variety\footnote{For a connected smooth proper \(Y/\ol k\), there is in fact a stronger sheaf-level decomposition \(\cH^1_Y(\Ql(1))\cong \cO_Y^\times\otimes\Ql\oplus \underline{\rH^1_{\et}(Y,\Ql(1))}\). For disconnected \(Y\), this is read componentwise. We do not use this decomposition in the present proof, only the cohomological consequences stated below.} it is a basic instance of Bloch--Ogus theory and
\(\rH^1_{\Zar}(\ol X,\cH^1)\simeq \NS(\ol X)\otimes\Ql\), so the weight condition is vacuous.

The details of the descent argument are a point of departure from the argument over \(\bbC\). Bhatt--Scholze's \(v\)-descent for vector bundles \cite[Theorem~4.1(ii)]{BhattScholze17} supplies rational descent for line bundles, a mechanism specific to characteristic \(p\). We then reduce to the case of seminormal proper schemes. In this case the degree-\(0\) descent for units is supplied separately by the sdh
topology. A final diagram with the auxiliary sheaf \(\cF\) of \S\ref{sec:auxiliary-sheaf-F}, cohomological descent, and weight separation then
identifies \(\rH^1_{\Zar}(\ol X, \cH^1)^{w=0}\) with \(\NS(\ol X)\otimes\Ql\).

\emph{Necessity of both conditions.} Restricting to Zariski-locally
trivial classes is essential even in the smooth proper setting. Indeed if \(\ol X\) is smooth proper over \(\ol k\), then all of \(\rH^2_{\et}(\ol X,\Ql(1))\) is pure of weight \(0\) \cite[Th\'{e}or\`{e}me~1.6]{Del80}. However the first Chern class map spans only the algebraic part
\(\NS(\ol X)\otimes\Ql\), whose rank (the Picard rank of \(\ol X\)) can be strictly smaller than the second Betti number \(b_{2}\). For a non-supersingular abelian surface the Picard rank is \(\le 4<6=b_{2}\), and for a finite-height K3 surface it is \(\le 20<22=b_{2}\). In either case \(\rH^2_{\et}(\ol X,\Ql(1))^{w=0}\) is strictly larger than \(\NS(\ol X)\otimes\Ql\).

The weight-zero condition is equally necessary. For example the complete-intersection surface \(X\subset\bbP^4\) of Barbieri-Viale and Srinivas \cite[Example~1]{BarbieriVialeSrinivas93}, a double cover of the cubic cone over an elliptic curve \(E\), branched along a quadric section, is normal and projective, yet \(\rH^1_{\Zar}(\ol X, \cH^1)\) acquires a nonzero weight-\((-1)\) summand \(\rH^1_{\et}(\ol E,\Ql(1))\) (Example~\ref{ex:normal-weight-example}).

\emph{Relation to the Tate conjecture.} Recall the Tate conjecture for divisors, that for \(\ol Y\) smooth and proper
the cycle class maps \(\NS(\ol Y)\otimes\Ql\) isomorphically onto the finite-order part\footnote{For a Frobenius
module \(V\), the \emph{finite-order part} \(V_{\mathrm{fo}}:=\bigcup_{n\ge1}\ker(F^{n}-\mathrm{id}_V)\) is the sum
of the eigenspaces whose eigenvalues are roots of unity. It satisfies \(V_{\mathrm{fo}}\subseteq V^{w=0}\)
(\ref{para:finite-order-part}).} \(\rH^2_{\et}(\ol Y,\Ql(1))_{\mathrm{fo}}\). Assuming the conjecture,
Theorem~\ref{thm:main-result} carries this identification to every proper reduced \(\ol X\):
\[
  \NS(\ol X)\otimes\Ql=\rH^1_{\Zar}(\ol X, \cH^1)_{\mathrm{fo}}=\rH^2_{\et}(\ol X,\Ql(1))_{\mathrm{fo}}
\]
(Proposition~\ref{prop:tate-divisors-fo}), where the first identification is unconditional.

\section{Setup and Preliminaries}
\label{sec:setup}

\begin{para}[Conventions]\label{para:conventions}
  Throughout this article we fix a finite field \(k=\bbF_{q}\) of characteristic \(p\) and denote by
  \(\overline{k}\) an algebraic closure of \(k\). We also fix a prime \(\ell \neq p\). When we state a
  result over a base more general than \(k\) or \(\ol{k}\), we use the letter \(K\) for the base field.
  From \S\ref{sec:dlog-smooth-proper} onward the base is either \(k\) or \(\ol{k}\).

  Unless otherwise stated, all schemes are assumed to be of finite type\footnote{The only exception to
  this will be Picard schemes.} and separated over the chosen base field. A variety is a geometrically
  integral scheme. For a scheme \(X/k\), we denote by \(\ol{X}\) the base change of \(X\) to \(\ol{k}\).

  For \(\tau\in\{\Zar,\et,\proet\}\), write \(X_{\tau}\) for the small \(\tau\)-site of \(X\), in the
  standard sense of \cite[Tag~020K]{stacks-project} (which also treats the big sites used below), whose
  objects are the schemes \(\tau\)-local over \(X\), namely the Zariski-open, \'{e}tale, and pro-\'{e}tale
  \(X\)-schemes respectively (for \(\proet\), see \cite{BS15}). The \(\sdh\),
  \(h\), and \(v\)-topologies are used on big sites. For
  \(\tau\in\{\sdh,h,v\}\) and \(X\) fixed, \(X_{\tau}\) denotes the big \(\tau\)-site, whose underlying
  category is that of finite-type \(X\)-schemes. The \(\sdh\)-topology is recalled in
  \S\ref{sec:sdh-seminormality}, and the \(h\)- and \(v\)-topologies enter the Picard-descent argument.
  All sheaf cohomology in this paper is taken on the displayed site. When a sheaf is constructed on a big
  site, we tacitly restrict to the relevant induced site. This restriction preserves the sheaf property.
\end{para}

\begin{para}[Tensoring Zariski sheaves with \(\Ql\)]\label{para:tensor-Ql-sheaves}
Let \(T\) be a noetherian scheme and let \(\mathcal A\) be an abelian sheaf on \(T_{\Zar}\). We write
\(\mathcal A\otimes\Ql\) for \(\mathcal A\otimes_{\bbZ}\underline{\Ql}\), where \(\underline{\Ql}\) is
the constant Zariski sheaf. For every open \(V\subset T\), we claim that
\[
  \Gamma(V,\mathcal A\otimes\Ql)
  \cong
  \Gamma(V,\mathcal A)\otimes_{\bbZ}\Ql.
\]
To see this, note that \(\Ql\) is the filtered colimit of its finitely generated subgroups \(N\subset \Ql\), each of
which is finite free over \(\bbZ\). Since sheafification preserves colimits, we have
\(\underline{\Ql}\cong\varinjlim_N\underline N\), and hence
\[
  \mathcal A\otimes\Ql
  \cong
  \varinjlim_N(\mathcal A\otimes_{\bbZ}N).
\]
For \(N\cong\bbZ^r\), one has \(\mathcal A\otimes_{\bbZ}N\cong\mathcal A^{\oplus r}\). Since \(V\) is
noetherian, sections over \(V\) commute with filtered colimits of Zariski sheaves by
\cite[Tag~0738]{stacks-project}, giving the claimed isomorphism on sections.

Tensoring with \(\Ql\) preserves flasque Zariski sheaves on noetherian schemes. For \(W\subset V\), the
restriction map after tensoring is identified with \(\Gamma(V,\mathcal A)\otimes\Ql\to\Gamma(W,\mathcal
A)\otimes\Ql\), hence is surjective if \(\mathcal A\) is flasque.
\end{para}

\begin{lem}\label{lem:pic-via-cartier}
Let \(X\) be a reduced noetherian scheme. Then
\[
  \rH^{0}_{\Zar}(X,\cO_{X}^{\times}\otimes\Ql)=\Gamma(X,\cO_{X}^{\times})\otimes\Ql,
  \qquad
  \rH^{1}_{\Zar}(X,\cO_{X}^{\times}\otimes\Ql)=\Pic(X)\otimes\Ql.
\]
If furthermore \(X\) is regular, then \(\rH^{q}_{\Zar}(X,\cO_{X}^{\times}\otimes\Ql)=0\) for \(q\geq 2\).
\end{lem}

\begin{proof}
The first equality follows from (\ref{para:tensor-Ql-sheaves}) applied to \(\cO_X^\times\), since
\[
  \rH^{0}_{\Zar}(X,\cO_{X}^{\times}\otimes\Ql)
  =\Gamma(X,\cO_{X}^{\times}\otimes\Ql)
  =\Gamma(X,\cO_{X}^{\times})\otimes\Ql.
\]

Let \(\cM_{X}\) denote the sheaf of total fractions on \(X\), obtained by sheafifying
\(U\mapsto\operatorname{Frac}(\cO_{X}(U))\). Its stalks are
\[
  \cM_{X,x}=\prod_{\eta\rightsquigarrow x}\cO_{X,\eta},
\]
the product over the generic points of the irreducible components passing through \(x\). The sheaf of
units \(\cM_{X}^{\times}\) is flasque on the noetherian reduced \(X\)
\cite[Tag~01X1]{stacks-project}, and so is \(\cM_{X}^{\times}\otimes\Ql\) by
(\ref{para:tensor-Ql-sheaves}). The short exact sequence of Zariski sheaves of abelian groups
\[
  0\longrightarrow\cO_{X}^{\times}\longrightarrow\cM_{X}^{\times}\longrightarrow\cD_{X}\longrightarrow 0,
\]
where \(\cD_{X}=\cM_{X}^{\times}/\cO_{X}^{\times}\) is the sheaf of Cartier divisors, remains exact after
tensoring with \(\Ql\). Its long exact sequence in cohomology, together with
\(\rH^{\geq 1}_{\Zar}(X,\cM_{X}^{\times}\otimes\Ql)=0\), gives
\[
  \rH^{1}_{\Zar}(X,\cO_{X}^{\times}\otimes\Ql)=\coker\bigl(\Gamma(X,\cM_{X}^{\times})\otimes\Ql\to\Gamma(X,\cD_{X})\otimes\Ql\bigr),
\]
where (\ref{para:tensor-Ql-sheaves}) identifies \(\Gamma(X,-\otimes\Ql)\) with \(\Gamma(X,-)\otimes\Ql\)
on the noetherian \(X\). Since \(\cM_X^\times\) is flasque, the untensored long exact sequence gives
\[
  \coker\bigl(\Gamma(X,\cM_{X}^{\times})\to\Gamma(X,\cD_{X})\bigr)
  \cong \rH^{1}_{\Zar}(X,\cO_{X}^{\times})=\Pic(X).
\]
Flatness of \(\Ql\) preserves cokernels, yielding
\(\rH^{1}_{\Zar}(X,\cO_{X}^{\times}\otimes\Ql)=\Pic(X)\otimes\Ql\).

Suppose now that \(X\) is regular. Each connected component of \(X\) is then a regular integral scheme,
and the higher-vanishing statement is local on the connected components, so we may assume \(X\) integral.
Weil and Cartier divisors on \(X\) then agree, so \(\cD_{X}=\Div_{X}=\bigoplus_{x\in
X^{(1)}}(i_{x})_{*}\bbZ\), where \(i_{x}:\Spec\kappa(x)\to X\). The sheaf \(\Div_{X}\) is flasque, and so
is \(\Div_{X}\otimes\Ql\) by (\ref{para:tensor-Ql-sheaves}). The tensored short exact sequence is then a
flasque resolution of \(\cO_{X}^{\times}\otimes\Ql\) of length two, so
\(\rH^{q}_{\Zar}(X,\cO_{X}^{\times}\otimes\Ql)=0\) for \(q\geq 2\).
\end{proof}

\begin{para}[Pro-\'{e}tale formalism]\label{para:proet}
We work in the pro-\'{e}tale site \cite{BS15} in order to treat \(\Ql(1)\) as a sheaf before pushing it
to the Zariski site. For a scheme \(X/k\) we denote by \(\lambda_{X}: X_\proet \to X_\Zar\) (resp.\
\(\lambda_{\ol{X}}:\ol{X}_{\proet} \to \ol{X}_{\Zar}\)) the canonical morphism of
sites\footnote{We will drop the subscripts from this and other notations when the context is clear.}. On
\(X_\proet\) and \(\ol{X}_{\proet}\) we have the sheaves \(\mu_{\ell^{n}}\) for each \(n \geq 1\).
Following \cite[Remark~5.2.4]{BS15}, we set \(\widehat{\bbZ}_{\ell}(1):=\varprojlim_{n}\mu_{\ell^{n}}\),
the inverse limit formed in the category of pro-\'{e}tale sheaves. The hat records that this sheaf is a
genuine inverse limit and not the sheafification of a constant presheaf, formation of constant sheaves
being incompatible with inverse limits \cite[Example~4.2.10]{BS15}. We set
\(\bbQ_{\ell}(1):=\widehat{\bbZ}_{\ell}(1)[\ell^{-1}]\). Finally, we have Zariski sheaves
\[
\cH^{q}_{X}(\bbQ_{\ell}(1)):=R^{q}\lambda_{*}\bbQ_{\ell}(1)
\qquad\text{and}\qquad
\cH^{q}_{\ol{X}}(\bbQ_{\ell}(1)):=R^{q}\lambda_{*}\bbQ_{\ell}(1).
\]
When we work over a finite or an algebraically closed field, the pro-\'{e}tale cohomology of
\(\bbQ_{\ell}(1)\) (and Jannsen's continuous \'{e}tale cohomology) agree with the usual \'{e}tale
cohomology with \(\bbQ_{\ell}(1)\)-coefficients \cite[(0.2) and Remark~3.5(c)]{Jannsen88},
\cite[Proposition 5.6.2]{BS15}.
In particular, \(\cH^{q}_{X}(\bbQ_{\ell}(1))\) (resp.
\(\cH^{q}_{\ol{X}}(\bbQ_{\ell}(1))\)) are sheaves associated to the presheaf
\(U \mapsto \rH^{q}_{\et}(U,\bbQ_{\ell}(1))\) (resp.
\(V \mapsto \rH^{q}_{\et}(V,\bbQ_{\ell}(1))\)) on \(X_{\Zar}\) (resp. \(\ol{X}_{\Zar}\)).
\end{para}

\begin{lem}\label{lem:h1-proet-presheaf-is-sheaf}
Assume \(K\) is perfect, and let \(X\) be a finite disjoint union of varieties over \(K\). Writing
\(\cH^1_X(\Ql(1))=R^1\lambda_{X,*}\Ql(1)\), for every open \(U\subset X\) the Leray edge map for
\(\lambda_U:U_{\proet}\to U_{\Zar}\) gives a functorial isomorphism
\[
  \rH^{1}_{\proet}(U,\Ql(1))
  \xrightarrow{\sim}
  \Gamma(U,\cH^{1}_{X}(\Ql(1))).
\]
In particular, the presheaf \(U\longmapsto \rH^{1}_{\proet}(U,\Ql(1))\) on \(X_{\Zar}\) is already a
Zariski sheaf.
\end{lem}

\begin{proof}
The assertion is componentwise, so we may assume \(X\) is a variety. Let \(U\) be a nonempty open subset
of \(X\). For every nonempty open \(V\subset U\), since \(V\) is geometrically connected over \(K\), we
have
\[
  \rH^{0}_{\proet}(V,\Ql(1))=\Bigl(\bigl(\varprojlim_{n}\mu_{\ell^{n}}(\ol{K})\bigr)\otimes_{\bbZ_{\ell}}\Ql\Bigr)^{G_{K}},
\]
the \(G_{K}\)-invariants of a one-dimensional representation, where \(G_{K}\) is the absolute Galois
group of \(K\). The value, either zero or one-dimensional, is independent of \(V\). Thus
\(R^{0}\lambda_{U,*}\Ql(1)\) is a constant sheaf on the irreducible space \(U_{\Zar}\), and is flasque. Every nonempty
open of \(U_{\Zar}\) is irreducible, so every restriction map between nonempty opens is the identity.

The Leray spectral sequence
\[
  E_{2}^{a,b}=\rH^{a}_{\Zar}(U,R^{b}\lambda_{U,*}\Ql(1)) \Longrightarrow \rH^{a+b}_{\proet}(U,\Ql(1))
\]
therefore has \(E_{2}^{1,0}=E_{2}^{2,0}=0\). The five-term exact sequence gives
\[
  \rH^{1}_{\proet}(U,\Ql(1))
  \cong
  \rH^{0}_{\Zar}(U,R^{1}\lambda_{U,*}\Ql(1))
  =
  \Gamma(U,\cH^{1}_{X}(\Ql(1))).
\]
These isomorphisms are functorial in \(U\), so the presheaf is identified with the Zariski sheaf
\(\cH^{1}_{X}(\Ql(1))\).
\end{proof}

\begin{para}[Galois action]\label{para:galois}
Let \(X/k\) be a scheme. The Galois group \(\rG_{k}:=\Gal(\ol{k}/k)\) acts on \(\ol{X}\) via its action
on \(\Spec(\ol{k})\), and this action extends to actions on the sites \(\ol{X}_{\proet}\) and
\(\ol{X}_{\Zar}\), thereby inducing an action on \(\cH^{q}_{\ol{X}}(\bbQ_{\ell}(1))\). In particular
\(\rG_{k}\) acts on \(\rH^{p}(\ol{X},\cH^{q}_{\ol{X}}(\bbQ_{\ell}(1)))\), and we denote by \(F_{k}\) the
corresponding action of the geometric Frobenius. The same considerations apply to the Leray spectral
sequence
\begin{equation}\label{eq:base-leray-SS}
  E_{2}^{p,q}=\rH^{p}_{\Zar}(\ol{X},\cH^{q}_{\ol{X}}(\bbQ_{\ell}(1)))\Longrightarrow
  \rH^{p+q}_{\proet}(\ol{X},\bbQ_{\ell}(1))
\end{equation}
for \(\lambda_{\ol{X}}\). In particular the differentials, edge maps, and the abutment filtration are
\(F_{k}\)-equivariant.
\end{para}

\begin{defn}[Frobenius module]\label{defn:Frobenius-module}
A Frobenius module is a finite-dimensional \(\bbQ_{\ell}\)-vector space \(V\) together with a
\(\bbQ_{\ell}\)-linear automorphism \(F\) all of whose eigenvalues are \(q\)-Weil numbers\footnote{A
\(q\)-Weil number is an algebraic number \(\alpha\) such that for every complex embedding of
\(\overline{\bbQ}\), the absolute value of the image of \(\alpha\) equals \(q^{w(\alpha)/2}\) for some
integer \(w(\alpha)\) independent of the chosen embedding. We call \(w(\alpha)\) the \emph{weight} of
\(\alpha\).}. A morphism of Frobenius modules is a \(\bbQ_{\ell}\)-linear map which respects the
\(F\)-action.
\end{defn}

Call \(V\) \textit{pure} of weight \(w\) if all eigenvalues are \(q\)-Weil numbers of weight
\(w\). The zero module is pure of every weight. Write \(V^{w=w_{0}}\) for the largest \(F\)-stable subspace
all of whose eigenvalues have weight \(w_{0}\).

\begin{lem}\label{lem:Frobenius-module}
Let \(V\) be a Frobenius module. Then
\begin{enumerate}[label=\textup{(\roman*)}]
\item \textup{(Weight decomposition)} \(V\) is canonically a direct sum of its pure
      submodules\footnote{The pure submodules need not be semisimple.}.
      \item \textup{(Subquotient)}
      Every subquotient (as a Frobenius module) of a pure module of weight \(w\) is pure of weight \(w\).
\item \textup{(Vanishing)} There are no non-zero maps between pure Frobenius modules of different
      weights.
\item \textup{(Exactness of weight-grading)} The functor \(V \mapsto V^{w=w_0}\) is exact.
\end{enumerate}

\end{lem}

\begin{proof}
Write \(F\) for the Frobenius automorphism of \(V\). It is pure of weight \(w\) when every eigenvalue
of \(F\) in \(\ol{\Ql}\) is a Weil \(q\)-number of weight \(w\) (absolute value \(q^{w/2}\) under every
complex embedding). Let \(P_{a}\in\Ql[t]\) be the product of the primary factors of the characteristic
polynomial whose roots have weight \(a\). The \(P_{a}\) are pairwise coprime, and primary decomposition
gives
\[
  V=\bigoplus_{a}\ker P_{a}(F).
\]
This is the decomposition in (i), and it is functorial in \(V\). A sub- or quotient module has
\(F\)-eigenvalues among those of \(V\), giving (ii). A morphism of Frobenius modules has image a subquotient
of both source and target, so a morphism between pure modules of distinct weights has image pure of two
weights, hence \(0\). This is (iii). Finally a morphism of Frobenius modules respects the decomposition
(i), so \(V\mapsto V^{w=w_{0}}\) carries a short exact sequence to the direct-summand sequence of
weight-\(w_{0}\) parts, which is exact. This is (iv).
\end{proof}

\begin{lem}\label{lem:Zar-Ql-weight}
Let \(T\) be a proper finite-type \(k\)-scheme. Then \(\rH^{p}_{\Zar}(\ol{T},\Ql(1))\) is a
Frobenius module pure of weight \(-2\) for every \(p\geq 0\).
\end{lem}

\begin{proof}
The constant Zariski sheaf \(\Ql(1)\) and its cohomology depend only on the underlying topological space
of \(\ol T\), so nilpotents play no role and we may pass to \(\ol T_{\mathrm{red}}\). After replacing the
geometric Frobenius \(F\) by a positive power, which does not change weights, we may assume that every
irreducible component of \(\ol T\) is \(F\)-stable. We argue by lexicographic induction on
\((\dim\ol T,\#\{\text{irreducible components of }\ol T\})\).

If \(\ol T\) is irreducible, then the constant Zariski sheaf \(\Ql(1)\) is flasque on the irreducible
noetherian space \(\ol T\). Hence the cohomology is \(\Ql(1)\) in degree \(0\) and \(0\) in higher degrees.
It is pure of weight \(-2\).

Otherwise write \(\ol T=\ol T'\cup \ol T_n\), where \(\ol T_n\) is one irreducible component and
\(\ol T'\) is the union of the remaining components. Put \(\ol S=\ol T'\cap\ol T_n\), and let
\(i'\colon\ol T'\hookrightarrow\ol T\), \(i_n\colon\ol T_n\hookrightarrow\ol T\), and
\(j\colon\ol S\hookrightarrow\ol T\) be the closed immersions. There is a short exact sequence
\[
  0\to\Ql(1)_{\ol T}\to
  i'_*\Ql(1)_{\ol T'}\oplus i_{n,*}\Ql(1)_{\ol T_n}
  \to j_*\Ql(1)_{\ol S}\to0,
\]
with the difference map on \(\ol S\). Exactness is checked on stalks. Away from \(\ol S\) it is the
identity sequence, and on \(\ol S\) it is
\[
  0\to\Ql(1)\xrightarrow{\mathrm{diag}}\Ql(1)^2
  \xrightarrow{\mathrm{diff}}\Ql(1)\to0.
\]
The associated long exact cohomology sequence filters \(\rH^p_{\Zar}(\ol T,\Ql(1))\) by subquotients of
\(\rH^{p-1}_{\Zar}(\ol S,\Ql(1))\) and
\(\rH^p_{\Zar}(\ol T',\Ql(1))\oplus\rH^p_{\Zar}(\ol T_n,\Ql(1))\). These groups are pure of weight \(-2\)
by induction. Lemma~\ref{lem:Frobenius-module} gives the same conclusion for
\(\rH^p_{\Zar}(\ol T,\Ql(1))\).
\end{proof}

\begin{lem}\label{lem:torsion-ag}
Let \(G\) be a commutative algebraic group of finite type over \(\ol{k}=\overline{\bbF_{q}}\). Then
\(G(\ol{k})\otimes_{\bbZ}\bbQ_{\ell}=0\).
\end{lem}
\begin{proof}
Choose a finite type model \(G_{0}/\bbF_{q^{n_{0}}}\) with \(G=G_{0}\times_{\bbF_{q^{n_{0}}}}\ol{k}\).
Then \(G(\ol{k})=\varinjlim_{n\geq n_{0}} G_{0}(\bbF_{q^{n}})\), and each \(G_{0}(\bbF_{q^{n}})\) is a
finite abelian group. Hence \(G(\ol{k})\) is a torsion \(\bbZ\)-module, and
\(G(\ol{k})\otimes_{\bbZ}\Ql=0\).
\end{proof}

\begin{para}[N\'eron-Severi group]\label{para:neron-severi}
For a proper scheme \(Y\) over \(\ol{k}\), let \(\Pic_{Y/\ol{k}}\) denote the Picard scheme
\cite[Exp.~XII]{SGA6} (see also \cite{Murre64} for an account over a field). This is a commutative
algebraic group locally of finite type over \(\ol{k}\), and its identity component
\(\Pic^{0}_{Y/\ol{k}}\) is connected and of finite type. We write
\[
  \Pic(Y) \;:=\; \Pic_{Y/\ol{k}}(\ol{k})
  \qquad\text{and}\qquad
  \Pic^{0}(Y) \;:=\; \Pic^{0}_{Y/\ol{k}}(\ol{k})
\]
for the abelian groups of \(\ol{k}\)-rational points of the Picard scheme and its identity component. By
representability of the Picard scheme, \(\Pic(Y)\) coincides with \(\rH^{1}_{\Zar}(Y,\cO_{Y}^{\times})\).
The \emph{N\'eron--Severi group} of \(Y\) is
\[
\NS(Y) \;:=\; \Pic(Y)\,\big/\,\Pic^{0}(Y),
\]
which by \cite[Exp.~XIII, Th\'{e}or\`{e}me~5.1]{SGA6} is finitely generated. By
Lemma~\ref{lem:torsion-ag} we have \(\Pic^{0}(Y)\otimes\Ql = 0\), and hence

\begin{equation}\label{eq:neron-severi-1}
\Pic(Y) \otimes \Ql \;=\; \NS(Y) \otimes \Ql.
\end{equation}

Now fix a proper scheme \(X\) over \(k\) and apply the above to \(Y = \ol{X}\). The geometric Frobenius
\(F\) acts on \(\NS(\ol{X})\) by pullback of line bundles. Each of the finitely many generators of
\(\NS(\ol{X})\) is represented by a line bundle defined over some finite extension \(\bbF_{q^{n_i}}\),
hence is \(F^{n_i}\)-invariant. Taking \(N = \mathrm{lcm}_i(n_i)\), we have \(F^N = \mathrm{id}\) on
\(\NS(\ol{X})\). Consequently the eigenvalues of \(F\) on \(\NS(\ol{X})\otimes\Ql\) are \(N\)-th roots of
unity, so \(\NS(\ol{X})\otimes\Ql\) is pure of weight~\(0\) as a Frobenius module.
\end{para}

\section{Main Result}
\label{sec:main-result}

For any scheme \(T\), recall the Zariski sheaf \(\cH^{1}=\cH^{1}_{T}(\bbQ_{\ell}(1))\) of
(\ref{para:proet}). We consider the Zariski cohomology group
\begin{equation}\label{eq:Zariski-local-cohomology}
\rH^{1}_{\Zar}(T,\cH^{1}).
\end{equation}
The next lemma identifies it with a graded piece of the pro-étale Leray filtration.

\begin{lem}\label{lem:leray-edge}
Assume \(K\) is finite or algebraically closed, and let \(Y\) be a finite disjoint union of varieties over \(K\). The Leray edge
for \(\lambda_{Y}:Y_{\proet}\to Y_{\Zar}\) gives an injection
\[
  \iota_{Y}:\;\rH^1_{\Zar}(Y, \cH^1)\;\lhook\joinrel\longrightarrow\;\rH^{2}_{\et}(Y,\Ql(1)).
\]
For \(Y=\ol{X}\) with \(X\) a finite disjoint union of varieties over \(k\), the injection
\(\iota_{\ol{X}}\) is \(G_{k}\)-equivariant.
\end{lem}

\begin{proof}
Write \(F^{0}\supseteq F^{1}\supseteq F^{2}\supseteq F^{3}=0\) for the abutment filtration on
\(\rH^{2}_{\et}(Y,\Ql(1))\), so that \(F^{p}/F^{p+1}=E_{\infty}^{p,2-p}\). It suffices to show
\(F^{2}=0\) and \(E_{2}^{1,1}=E_{\infty}^{1,1}\). Then \(F^{1}=E_{\infty}^{1,1}=\rH^1_{\Zar}(Y, \cH^1)\) and
the inclusion \(F^{1}\hookrightarrow\rH^{2}_{\et}(Y,\Ql(1))\) is the asserted \(\iota_{Y}\).

\emph{Vanishing of \(F^{2}\).} By the flasqueness argument in Lemma~\ref{lem:h1-proet-presheaf-is-sheaf},
\(R^{0}\lambda_{Y,*}\Ql(1)\) is locally constant on \(Y_{\Zar}\), hence flasque. Thus \(\rH^{p}_{\Zar}(Y,R^{0}\lambda_{Y,*}\Ql(1))=0\) for \(p\geq
1\). In particular \(E_{2}^{2,0}=0\). Since \(F^{3}=0\), \(F^{2}=E_{\infty}^{2,0}\) is a subquotient of
\(E_{2}^{2,0}\), hence zero.

\emph{\(E_{2}^{1,1}=E_{\infty}^{1,1}\).} The differential \(d_{2}:E_{2}^{1,1}\to E_{2}^{3,0}\) targets
\(\rH^{3}_{\Zar}(Y,R^{0}\lambda_{Y,*}\Ql(1))=0\) by the same flasqueness. For \(r\geq 3\), incoming
differentials \(d_{r}:E_{r}^{1-r,r}\to E_{r}^{1,1}\) have first index \(<0\), and outgoing
\(d_{r}:E_{r}^{1,1}\to E_{r}^{1+r,2-r}\) have second index \(<0\), so source or target vanishes.

The \(G_{k}\)-equivariance for \(Y=\ol{X}\) follows from (\ref{para:galois}). The Leray spectral sequence
and its abutment filtration consist of \(G_{k}\)-modules.
\end{proof}

\begin{prop}\label{prop:ZL-frobenius-comparison}
Let \(X\) be a proper reduced \(k\)-scheme. Then \(\rH^1_{\Zar}(\ol X,\cH^1)\) is a Frobenius module.
Moreover there is a canonical map
\[
  \theta_X:F^1\rH^2_{\et}(\ol X,\Ql(1))\longrightarrow \rH^1_{\Zar}(\ol X, \cH^1)
\]
whose kernel and cokernel are pure of weight \(-2\). Consequently \(\theta_X\) induces an isomorphism
\[
  \bigl(F^1\rH^2_{\et}(\ol X,\Ql(1))\bigr)^{w=0}
  \xrightarrow{\ \sim\ }
  \rH^1_{\Zar}(\ol X, \cH^1)^{w=0}.
\]
If \(\ol X\) is a finite disjoint union of varieties, then \(\theta_X\) is an isomorphism. Equivalently,
\(\rH^1_{\Zar}(\ol X, \cH^1)\) identifies with the subspace \(F^1\rH^2_{\et}(\ol X,\Ql(1))\).
\end{prop}

\begin{proof}
Consider the Leray spectral sequence for
\(\lambda_{\ol X}:\ol X_{\proet}\to\ol X_{\Zar}\):
\[
  E_2^{a,b}=\rH^a_{\Zar}\bigl(\ol X,\cH^b_{\ol X}(\Ql(1))\bigr)
  \Longrightarrow \rH^{a+b}_{\et}(\ol X,\Ql(1)).
\]
It is Frobenius-equivariant by \ref{para:galois}. The sheaf \(R^0\lambda_{\ol X,*}\Ql(1)\) is the constant
Zariski sheaf \(\Ql(1)\). For every Zariski open \(U\subset \ol X\),
\(\rH^0_{\proet}(U,\Ql(1))=\Ql(1)^{\pi_0(U)}\). Hence
\[
  E_2^{a,0}=\rH^a_{\Zar}(\ol X,\Ql(1)),
\]
and Lemma~\ref{lem:Zar-Ql-weight} says that these groups are pure of weight \(-2\).

Let \(F^\bullet\) be the abutment filtration on \(\rH^2_{\et}(\ol X,\Ql(1))\). The abutment is a
Frobenius module (in general mixed) by \cite[Th\'{e}or\`{e}me~1.6]{Del80}. Therefore each graded piece
\(E^{a,2-a}_\infty\) has Weil-number Frobenius eigenvalues. At \((1,1)\) there
are no incoming differentials by indexing. Thus
\[
  E^{1,1}_\infty=\ker\bigl(d_2:E_2^{1,1}\to E_2^{3,0}\bigr),
\]
and \(E_2^{1,1}\) is an extension of \(\operatorname{im}d_2\) by \(E^{1,1}_\infty\). Since
\(\operatorname{im}d_2\) is a subquotient of the pure weight-\((-2)\) group \(E_2^{3,0}\), this proves that
\(E_2^{1,1}=\rH^1_{\Zar}(\ol X, \cH^1)\) is a Frobenius module.

The map \(\theta_X\) is the composite
\[
  F^1\rH^2_{\et}(\ol X,\Ql(1))
  \twoheadrightarrow E^{1,1}_\infty
  \hookrightarrow E_2^{1,1}=\rH^1_{\Zar}(\ol X, \cH^1).
\]
Its kernel is \(F^2\rH^2_{\et}(\ol X,\Ql(1))=E^{2,0}_\infty\), a subquotient of
\(E_2^{2,0}=\rH^2_{\Zar}(\ol X,\Ql(1))\), hence pure of weight \(-2\) by
Lemma~\ref{lem:Zar-Ql-weight}. Its cokernel is \(\operatorname{im}d_2\subset E_2^{3,0}\), again pure of
weight \(-2\). The functor \(V\mapsto V^{w=0}\) is exact (Lemma~\ref{lem:Frobenius-module}), so
\(\theta_X\) becomes an isomorphism on weight-zero parts.

If \(\ol X\) is a finite disjoint union of varieties, Lemma~\ref{lem:leray-edge} gives
\(F^2=0\) and \(d_2=0\). Hence the displayed map \(\theta_X\) is already an isomorphism.
\end{proof}

\begin{para}[The \(\ell\)-adic cycle class map]\label{para:cycle-class}
Let \(K\) be a field of characteristic different from \(\ell\) and let \(X\) be a \(K\)-scheme. For each
\(n\geq1\) the Kummer sequence of \'{e}tale
sheaves on \(X_{\et}\),
\[
  0\longrightarrow \mu_{\ell^{n}}\longrightarrow \bbG_{m}\xrightarrow{\;(-)^{\ell^{n}}\;}\bbG_{m}\longrightarrow 0,
\]
forms an inverse system in \(n\) under the \(\ell\)-th-power transition maps. By
\cite[Lemma~3.26]{Jannsen88} (see also \cite[\S 6.1.1, eqs.~(6.2)--(6.3)]{Schreieder23}), passing to
Jannsen's continuous \'{e}tale cohomology (equivalently pro-\'{e}tale, see~(\ref{para:proet})),
sheafifying on \(X_{\Zar}\), and composing with the canonical map
\(R^{1}\lambda_{X,*}\widehat{\bbZ}_{\ell}(1)\to \cH^{1}_{X}(\Ql(1))\), the connecting maps assemble into
the \emph{\(\ell\)-adic logarithmic differential}\footnote{The notation \(d\log\) follows the Hodge-theoretic analogue in
\cite[\S 3, eq.~(15) and Lemma~4]{BRS09}.}
\[
  d\log_{\ell,X}:\cO_{X}^{\times}\longrightarrow \cH^{1}_{X}(\Ql(1)).
\]
Schreieder \cite[\S 6.1.1, eq.~(6.3) and Lemma~6.7]{Schreieder23} works with the resulting global Chern
class \(c_{1}\colon\Pic(X)\to\rH^{2}_{\et}(X,\bbZ_{\ell}(1))\). Applying \(\rH^{1}_{\Zar}(X,-)\)
and using \(\rH^{1}_{\Zar}(X,\cO_{X}^{\times})=\Pic(X)\) gives a \(\bbZ\)-linear map
\[
  \widetilde{c}_{1}:\Pic(X)\longrightarrow \rH^{1}_{\Zar}(X,\cH^1).
\]
Since the target is a \(\Ql\)-vector space, \(\widetilde{c}_{1}\) factors uniquely through
\(\Pic(X)\otimes\Ql\), defining the \emph{\(\ell\)-adic cycle class map}
\[
  \widetilde{c}_{1}:\Pic(X)\otimes\Ql\;\longrightarrow\;\rH^1_{\Zar}(X, \cH^1).
\]
For \(X=\ol{Y}\) the base change of a proper \(Y/k\), we have \(\Pic^{0}(\ol{Y})\otimes\Ql=0\) by
(\ref{para:neron-severi}). The cycle class then factors through \(\NS(\ol{Y})\otimes\Ql\):
\[
  \widetilde{c}_{1}:\NS(\ol{Y})\otimes\Ql\longrightarrow \rH^1_{\Zar}(\ol{Y}, \cH^1).
\]

\emph{Notation for the Chern class.} We write \(c_{1}\) for the \'{e}tale first Chern class, with target
\(\rH^{2}_{\et}(-,\bbZ_{\ell}(1))\), its rationalisation \(\rH^{2}_{\et}(-,\Ql(1))\), or a subgroup such as
the \(F^{1}\) of (\ref{para:integral-form}), and \(\widetilde{c}_{1}\) for its Zariski-sheafified avatar,
valued in the subquotient \(\rH^1_{\Zar}(-, \cH^1)\). Each is used both for the integral map and for its
\(\Ql\)-linearisation.

\emph{Functoriality and \(G_{k}\)-equivariance.} For any morphism \(f:X\to X'\) of \(K\)-schemes,
pullback of structure sheaves \(f^{-1}\cO_{X'}\to\cO_{X}\) is compatible with multiplication and
\(\ell^{r}\)-th-power maps. It therefore induces a morphism of Kummer inverse systems on \(X_{\et}\).
Passing through the continuous \'{e}tale cohomology functor and sheafifying on the Zariski sites
shows that \(d\log_{\ell,X}\), and hence \(\widetilde{c}_{1}\), are contravariantly functorial in
\(X\), so that for any such \(f\) the square
\[
\begin{array}{ccc}
  \Pic(X')\otimes\Ql & \xrightarrow{\;\widetilde{c}_{1}\;} & \rH^1_{\Zar}(X', \cH^1) \\
  \scriptstyle f^{*}\Big\downarrow & & \Big\downarrow\scriptstyle f^{*} \\
  \Pic(X)\otimes\Ql & \xrightarrow{\;\widetilde{c}_{1}\;} & \rH^1_{\Zar}(X, \cH^1)
\end{array}
\]
commutes. Specialising to \(X=X'=\ol{Y}=Y\times_{k}\ol{k}\) for \(Y/k\), apply this to the
\(k\)-automorphisms \(\sigma_{\ol{Y}}:=\mathrm{id}_{Y}\times\sigma:\ol{Y}\to\ol{Y}\) for \(\sigma\in
G_{k}=\Gal(\ol{k}/k)\). The two resulting \(G_{k}\)-actions, on \(\Pic(\ol{Y})\otimes\Ql\) (by pullback
of line bundles) and on \(\rH^1_{\Zar}(\ol{Y}, \cH^1)\) (from (\ref{para:galois})), make
\(\widetilde{c}_{1}\) a morphism of \(G_{k}\)-modules.
\end{para}

\begin{thm}\label{thm:main-result}
Let \(X\) be a proper reduced scheme over \(k=\bbF_{q}\). Then the \(\ell\)-adic cycle class map
\[
\widetilde{c}_{1}: \NS(\ol{X}) \otimes \bbQ_{\ell} \to \rH^1_{\Zar}(\ol{X}, \cH^1)^{w=0}
\]
is a \(\rG_{k}\)-equivariant isomorphism.

\end{thm}

We now deduce the integral refinement of Theorem~\ref{thm:main-result} stated in the introduction
(Corollary~\ref{cor:integral-form}). The proof of the theorem itself is completed in
\S\ref{sec:main-theorem-proof}, drawing on the contents of \S\S\ref{sec:sdh-seminormality}--\ref{sec:auxiliary-sheaf-F}.

\begin{para}[Integral form]\label{para:integral-form}
Theorem~\ref{thm:main-result} refines to an integral statement of Lefschetz \((1,1)\) type. For each
\(n\) the Kummer sequence of (\ref{para:cycle-class}) yields on \(\ol{X}_{\et}\) the exact sequence
\[
  0\to\Pic(\ol X)/\ell^{n}\to\rH^{2}_{\et}(\ol X,\mu_{\ell^{n}})\to\rH^{2}_{\et}(\ol X,\bbG_{m})[\ell^{n}]\to0 .
\]
Since \(\Pic^{0}(\ol X)(\ol k)\) is \(\ell\)-divisible, we have \(\Pic(\ol X)/\ell^{n}=\NS(\ol X)/\ell^{n}\). Further the transition maps \(\Pic(\ol X)/\ell^{n+1}\twoheadrightarrow\Pic(\ol X)/\ell^{n}\) are surjective, so \(\varprojlim^{1}_{n}\Pic(\ol X)/\ell^{n}=0\). Passing to
\(\varprojlim_{n}\) and using
\(\varprojlim_{n}\NS(\ol X)/\ell^{n}=\NS(\ol X)\otimes\bbZ_{\ell}\) (as \(\NS(\ol X)\) is finitely
generated), \(\varprojlim_{n}\rH^{2}_{\et}(\ol X,\mu_{\ell^{n}})=\rH^{2}_{\et}(\ol X,\bbZ_{\ell}(1))\)
(\ref{para:proet}), and \(\varprojlim_{n}\rH^{2}_{\et}(\ol X,\bbG_{m})[\ell^{n}]=T_{\ell}\,\mathrm{Br}(\ol
X)\), gives the exact sequence of \(\bbZ_{\ell}\)-modules
\begin{equation}\label{eq:integral-NS-ses}
  0\to\NS(\ol X)\otimes\bbZ_{\ell}\to\rH^{2}_{\et}(\ol X,\bbZ_{\ell}(1))\to T_{\ell}\,\mathrm{Br}(\ol X)\to0 ,
\end{equation}
with the \(\ell\)-adic Tate module of the Brauer group being torsion-free. Hence
\(\NS(\ol X)\otimes\bbZ_{\ell}\) is a saturated subgroup of \(\rH^{2}_{\et}(\ol X,\bbZ_{\ell}(1))\)
containing all of its torsion.

Recall that \(\rH^{2}_{\ZL}(\ol X,\bbZ_{\ell}(1))\), the subgroup of Zariski-locally trivial classes, is the
first step \(F^{1}\rH^{2}_{\et}(\ol X,\bbZ_{\ell}(1))\) of the Leray filtration for \(\lambda_{\ol X}\). A class
lies in \(F^{1}\) precisely when its image in \(\rH^{0}_{\Zar}(\ol X,R^{2}\lambda_{\ol X,*}\widehat{\bbZ}_{\ell}(1))\)
vanishes, that is, when it restricts to zero on some Zariski-open cover. The
cycle class refines to \(c_{1}\colon\NS(\ol X)\otimes\bbZ_{\ell}\to\rH^{2}_{\ZL}(\ol X,\bbZ_{\ell}(1))\).

The same filtered complex
\[
  R\Gamma_{\Zar}(\ol X,R\lambda_{\ol X,*}-)
\]
induces the integral and rational Leray filtrations, with coefficients \(\widehat{\bbZ}_{\ell}(1)\) and
\(\Ql(1)=\widehat{\bbZ}_{\ell}(1)[\ell^{-1}]\). Since inverting \(\ell\) is exact,
\[
  F^{1}\rH^{2}_{\et}(\ol X,\bbZ_{\ell}(1))\otimes\Ql
  =
  F^{1}\rH^{2}_{\et}(\ol X,\Ql(1))
\]
inside \(\rH^{2}_{\et}(\ol X,\Ql(1))\).

By Proposition~\ref{prop:ZL-frobenius-comparison}, the weight-\(0\) part of
\(F^{1}\rH^{2}_{\et}(\ol X,\Ql(1))\) equals \(\rH^1_{\Zar}(\ol X, \cH^1)^{w=0}\).

\begin{cor}\label{cor:integral-form}
For \(X\) proper reduced over \(k\) the cycle class is an isomorphism
\[
  c_{1}\colon\NS(\ol X)\otimes\bbZ_{\ell}\xrightarrow{\ \sim\ }
  \bigl\{\alpha\in\rH^{2}_{\ZL}(\ol X,\bbZ_{\ell}(1)) : \alpha\otimes\Ql\in\rH^{2}_{\et}(\ol X,\Ql(1))^{w=0}\bigr\}.
\]
In particular every torsion class of \(\rH^{2}_{\et}(\ol X,\bbZ_{\ell}(1))\) (necessarily \(\ell\)-primary) lies in the image and comes
from the \(\ell\)-primary torsion of the finitely generated group \(\NS(\ol X)\).
\end{cor}

\begin{proof}
The image of \(c_{1}\) is Zariski-locally trivial and pure of weight \(0\) rationally, so \(c_{1}\) maps
into the displayed set and is injective by \eqref{eq:integral-NS-ses}. Conversely, if \(\alpha\) lies in
the set then
\(\alpha\otimes\Ql\in\bigl(F^{1}\rH^{2}_{\et}(\ol X,\Ql(1))\bigr)^{w=0}=\rH^1_{\Zar}(\ol X, \cH^1)^{w=0}=\NS(\ol
X)\otimes\Ql\) by Theorem~\ref{thm:main-result}. Thus \(\alpha\otimes\Ql\) lies in the \(\Ql\)-span of
\(\NS(\ol X)\otimes\bbZ_{\ell}\), and saturation in \eqref{eq:integral-NS-ses} forces
\(\alpha\in\NS(\ol X)\otimes\bbZ_{\ell}\). The torsion assertion follows from torsion-freeness of the
\(\ell\)-adic Tate module of the Brauer group.
\end{proof}
\end{para}

\section{sdh hypercovers and seminormality}
\label{sec:sdh-seminormality}

We use the sdh topology to compute functions and units in degree \(0\) on a seminormal scheme from a
smooth proper hypercover.

Following \cite[Tag~0EUS]{stacks-project}, let
\(\rho\colon X^{\mathrm{sn}}\to X\) be the seminormalisation of \(X\). It is initial among universal
homeomorphisms \(Y\to X\) inducing isomorphisms on residue fields. A scheme is \emph{seminormal} if
\(\rho\) is an isomorphism. This condition is local on affine opens \cite[Tag~0EUP]{stacks-project}.

\begin{defn}[sdh-topology]\label{defn:sdh-topology}
For a field \(K\), the \emph{sdh-topology} on the category of \(K\)-schemes is the Grothendieck topology
generated by \'{e}tale covers together with proper morphisms \(f:Y\to T\) such that for every \(t\in T\)
there exists \(y\in f^{-1}(t)\) with \(k(y)/k(t)\) finite separable. For a \(K\)-scheme \(X\), the small
sdh-site \(X_{\sdh}\) is the induced site on \(X\)-schemes.
\end{defn}

In characteristic \(p\) the separability condition is essential. Dropping it produces the h-topology,
whose covers include Frobenius and whose sheafification sees the perfection rather than the
seminormalisation.

For a field \(K\), let \(a_{\sdh}\) denote sdh sheafification on \(K\)-schemes, and let \(h_{Y}\) denote
the presheaf represented by a \(K\)-scheme \(Y\).

\begin{prop}[Huber--Kelly]\label{prop:huber-kelly}
Let \(K\) be a field. For every reduced \(K\)-scheme \(U\),
\[
  (a_{\sdh}h_{\bbA^{1}})(U)=\Gamma(U^{\mathrm{sn}},\cO_{U^{\mathrm{sn}}}) \qquad\text{and}\qquad
  (a_{\sdh}h_{\bbG_{m}})(U)=\Gamma(U^{\mathrm{sn}},\cO^{\times}_{U^{\mathrm{sn}}}).
\]
\end{prop}

\begin{proof}
The statement for \(\bbA^{1}\) is \cite[Proposition~6.14]{HuberKelly2018} (whose Nagata hypothesis holds
for the field \(K\)). The sdh-sheafification of the functions presheaf is computed on the
seminormalisation. For \(\bbG_{m}\), realise \(\bbG_{m}\) as the closed subscheme
\(\{xy=1\}\subseteq\bbA^{2}\). Then \(h_{\bbG_{m}}\) is the
equaliser of the two maps \(h_{\bbA^{1}}\times h_{\bbA^{1}}\rightrightarrows h_{\bbA^{1}}\),
\((f,g)\mapsto fg\) and \((f,g)\mapsto 1\). Sheafification is left exact, so \(a_{\sdh}h_{\bbG_{m}}\) is the
equaliser of \((a_{\sdh}h_{\bbA^{1}})^{2}=\Gamma((-)^{\mathrm{sn}},\cO)^{2}\) under the same two maps,
namely \(\{(f,g):fg=1\}=\Gamma((-)^{\mathrm{sn}},\cO^{\times})\).
\end{proof}

An \emph{sdh-hypercover} of a \(K\)-scheme \(X\) is an augmented simplicial \(K\)-scheme
\(\pi_{\bullet}:X_{\bullet}\to X\) such that
\(X_{n}\to(\cosk_{n-1}X_{\bullet})_{n}\) is an sdh-cover for all \(n\geq0\), with the convention
\((\cosk_{-1}X_{\bullet})_{0}=X\). It is \emph{smooth proper} if each \(X_n\) is a finite disjoint union
of smooth proper \(K\)-schemes.

\begin{prop}\label{prop:smooth-proper-sdh-hypercover}
Let \(X\) be a proper \(k\)-scheme. Then there exists an sdh-hypercover \(\pi_{\bullet}:X_{\bullet}\to
X\) such that each \(X_{n}\) is a finite disjoint union of smooth proper \(k\)-schemes.
\end{prop}

\begin{proof}
Call a proper morphism \(Y\to T\) \emph{separably decomposed} if every \(t\in T\) has a preimage \(y\)
with \(k(y)/k(t)\) finite separable. Such morphisms are sdh-covers and are stable under base change,
composition, and finite disjoint union, and the defining condition is checked pointwise on \(T\). We
claim every proper \(T/k\) admits a separably decomposed cover by a finite disjoint union of smooth
proper \(k\)-schemes, arguing by noetherian induction on \(\dim T\). Replacing \(T\) by
\(T_{\mathrm{red}}\) (the nilimmersion \(T_{\mathrm{red}}\to T\) induces isomorphisms on residue fields,
hence is an sdh-cover) we may assume \(T\) reduced. Let \(T_{1},\dots,T_{r}\) be its irreducible components
and \(U_{i}\subseteq T_{i}\) a dense open disjoint from the other components. By
\cite[Theorem~4.1 and Remark~4.2]{deJong96} there is an alteration \(\widetilde T_{i}\to T_{i}\) with
\(\widetilde T_{i}\) smooth proper over \(k\). As \(k\) is perfect it may be taken generically \'etale, so
after shrinking \(U_{i}\) it is finite \'etale over \(U_{i}\). A finite \'etale morphism has finite
separable residue field extensions at every point, so \(\widetilde T_{i}\to T_{i}\) is separably
decomposed over \(U_{i}\). The closed complement \(T\setminus\bigsqcup_{i}U_{i}\) has strictly smaller
dimension, so by induction admits such a cover. The disjoint union of that cover with the
\(\widetilde T_{i}\) is a separably decomposed smooth proper cover of \(T\).

Now apply the split Verdier hypercover construction
\cite[V$_{\mathrm{bis}}$, Section 5]{SGA4_II}. Having chosen the lower non-degenerate levels, the
\(n\)-th matching object is a finite limit of schemes proper over \(X\), hence proper over \(k\). Cover
it by a finite disjoint union of smooth proper \(k\)-schemes as above and take this to be the
non-degenerate \(n\)-th level. The resulting split simplicial scheme has smooth proper levels and
matching maps which are sdh-covers.
\end{proof}

\begin{lem}\label{lem:sdh-degree-zero}
Let \(K\) be a field, let \(X\) be a reduced scheme over \(K\), let \(\rho:X^{\mathrm{sn}}\to X\) be its
seminormalisation, and let \(\pi_{\bullet}:X_{\bullet}\to X\) be a smooth sdh-hypercover over \(K\). Then
\[
  R^{0}\pi_{\bullet,*}\cO_{X_{\bullet}}\cong \rho_{*}\cO_{X^{\mathrm{sn}}} \qquad\text{and}\qquad
  R^{0}\pi_{\bullet,*}\cO^{\times}_{X_{\bullet}}\cong \rho_{*}\cO^{\times}_{X^{\mathrm{sn}}}
\]
as sheaves on \(X_{\Zar}\). In particular, if \(X\) is seminormal, then
\[
  R^{0}\pi_{\bullet,*}\cO_{X_{\bullet}}\cong \cO_{X} \qquad\text{and}\qquad
  R^{0}\pi_{\bullet,*}\cO^{\times}_{X_{\bullet}}\cong \cO^{\times}_{X}.
\]
Moreover, for \(G_{\bullet}=\cO_{X_{\bullet}}\) or \(\cO^{\times}_{X_{\bullet}}\), the natural map
\[
  \bigl(R^{0}\pi_{\bullet,*}G_{\bullet}\bigr)\otimes_{\bbZ}\Ql
  \xrightarrow{\;\sim\;}
  R^{0}\pi_{\bullet,*}\bigl(G_{\bullet}\otimes_{\bbZ}\Ql\bigr)
\]
is an isomorphism of sheaves of \(\Ql\)-vector spaces on \(X_{\Zar}\). Thus, for \(X\) seminormal,
\[
  R^{0}\pi_{\bullet,*}\bigl(\cO_{X_{\bullet}}\otimes\Ql\bigr)\cong \cO_{X}\otimes\Ql
  \qquad\text{and}\qquad R^{0}\pi_{\bullet,*}\bigl(\cO^{\times}_{X_{\bullet}}\otimes\Ql\bigr)\cong
  \cO^{\times}_{X}\otimes\Ql.
\]
\end{lem}

\begin{proof}
For a Zariski open \(U\subset X\), put \(U_i=U\times_X X_i\). Degree-zero simplicial direct image is the
equalizer
\[
  \Gamma(U,R^0\pi_{\bullet,*}G_\bullet)
  =
  \operatorname{Eq}\bigl(\Gamma(U_0,G_0)\rightrightarrows\Gamma(U_1,G_1)\bigr).
\]
Apply this to \(G_\bullet=\cO_{X_\bullet}\) and \(\cO^\times_{X_\bullet}\). For
\(F=a_{\sdh}h_{\bbA^1}\) or \(a_{\sdh}h_{\bbG_m}\), the sdh-sheaf condition for the hypercover
\(U_\bullet\to U\) gives
\[
  F(U)=\operatorname{Eq}\bigl(F(U_0)\rightrightarrows F(U_1)\bigr).
\]
This is obtained from the sheaf condition for \(U_0\to U\) and the level-one matching cover
\(U_1\to U_0\times_U U_0\). The \(U_i\) are smooth, hence seminormal, so Proposition
\ref{prop:huber-kelly} identifies \(F(U_i)\) with ordinary functions or units on \(U_i\). Applying the
same proposition to \(U\) identifies \(F(U)\) with functions or units on \(U^{\mathrm{sn}}\). This gives
the two displayed sheaf isomorphisms, and the seminormal case follows by \(U^{\mathrm{sn}}=U\).

For the rationalised assertion, tensor the displayed equalizer with \(\Ql\). The section formula
(\ref{para:tensor-Ql-sheaves}) and flatness of \(\Ql\) over \(\bbZ\) identify it with the equalizer
computing \(R^0\pi_{\bullet,*}(G_\bullet\otimes\Ql)\), functorially in \(U\).
\end{proof}

\section{The smooth-proper case}
\label{sec:dlog-smooth-proper}

\begin{prop}\label{cor:cycle-class-smooth-proper}
Let \(Y\) be a smooth proper scheme over \(k\), and set \(\ol Y=Y\times_k\ol k\). Then:
\begin{enumerate}[label=\textup{(\roman*)},nosep]
\item the cycle class of (\ref{para:cycle-class}) induces a \(G_k\)-equivariant isomorphism
\[
  \widetilde{c}_{1}:\NS(\ol Y)\otimes\Ql
  \xrightarrow{\sim}
  \rH^1_{\Zar}(\ol Y,\cH^1),
\]
and the source and target are pure of weight \(0\);
\item the Leray edge gives a functorial isomorphism
\[
  \rH^0_{\Zar}(\ol Y,\cH^1)
  \cong
  \rH^1_{\et}(\ol Y,\Ql(1)),
\]
pure of weight \(-1\);
\item \(\rH^r_{\Zar}(\ol Y,\cH^1)=0\) for \(r\ge2\).
\end{enumerate}
All identifications are contravariantly functorial for morphisms of smooth proper \(k\)-schemes.
\end{prop}

\begin{proof}
All assertions are additive over the connected components of \(\ol Y\), each smooth and irreducible,
hence equidimensional, so the equidimensional statements cited below apply componentwise. The pro-étale
groups \(\rH^*(\ol Y,\Ql(1))\) form an \(\ell\)-adic twisted
Borel--Moore cohomology theory \cite[Proposition~6.6]{Schreieder23} which agrees with continuous \'{e}tale
cohomology \cite[Lemma~6.5]{Schreieder23}. For such a theory, the Bloch--Ogus coniveau formalism
\cite{BlochOgus74} identifies \(\rH^p_{\Zar}(\ol Y,\cH^q_{\ol Y}(\Ql(1)))\) with the \(E_2^{p,q}\)-term
of the coniveau spectral sequence \cite[\S7.11]{Schreieder23}. For the row \(q=1\), the Gersten resolution exhibits \(\cH^1\) as the
kernel of the Kummer valuation residue \(\partial\) in the two-term flasque complex
\[
  \bigoplus_{\eta\in\ol Y^{(0)}}(i_{\eta})_{*}\rH^{1}_{\et}(\kappa(\eta),\Ql(1))\;\xrightarrow{\ \partial\ }\;\bigoplus_{y\in\ol Y^{(1)}}(i_{y})_{*}\Ql,
\]
the constant sheaves on the generic points and the codimension-one skyscrapers, both flasque. Taking
\(\rH^{*}_{\Zar}\) of this length-one flasque resolution gives
\(\rH^{r}_{\Zar}(\ol Y,\cH^1)=0\) for \(r\ge2\), which is (iii). The degree-zero term
\(\rH^{0}_{\Zar}=\ker\partial\) is treated in (ii) below, and
\(\rH^{1}_{\Zar}(\ol Y,\cH^1)=\operatorname{coker}\partial=E_{2}^{1,1}\). By the
Bloch--Ogus identification of the diagonal coniveau term \(E_{2}^{1,1}\) with codimension-one cycles
modulo rational equivalence \cite{BlochOgus74} (in the pro-\'etale form of \cite[\S7.11]{Schreieder23}),
this cokernel is \(\mathrm{CH}^{1}(\ol Y)\otimes\Ql=\Pic(\ol Y)\otimes\Ql\). Concretely \(\partial\) sends
the \(d\log\) class of a rational function \(f\in\kappa(\eta)^{\times}\) to its divisor
\(\operatorname{div}(f)\), so the cokernel is the divisor class group, rationally. Since
\(\Pic^{0}(\ol Y)\otimes\Ql=0\) (Lemma~\ref{lem:torsion-ag}), this is \(\NS(\ol Y)\otimes\Ql\). The edge map is the first Chern class,
coinciding with the normalisation \(\widetilde{c}_{1}\) of (\ref{para:cycle-class}) by
\cite[Lemma~3.26]{Jannsen88} and \cite[Lemma~6.7]{Schreieder23}. This is (i). Functoriality follows from
functoriality of the coniveau spectral sequence and from Kummer functoriality of the first Chern class.

For (ii), use the Leray spectral sequence for \(\lambda_{\ol Y}:\ol Y_{\proet}\to\ol Y_{\Zar}\). On each
connected component \(R^0\lambda_{\ol Y,*}\Ql(1)\) is the constant Zariski sheaf \(\Ql(1)\), hence
flasque. The low-degree edge sequence therefore gives
\[
  \rH^1_{\et}(\ol Y,\Ql(1))
  \xrightarrow{\sim}
  \rH^0_{\Zar}(\ol Y,\cH^1).
\]
The weight assertions are Deligne's Weil II purity for \(\rH^1_{\et}(\ol Y,\Ql(1))\) and the
weight \(0\) Frobenius action on \(\NS(\ol Y)\otimes\Ql\) (\ref{para:neron-severi}).
\end{proof}

\section{The simplicial cycle class}
\label{sec:simp-cycle-class}

Throughout this section, fix a proper \(k\)-scheme \(X\) and a smooth proper hypercover
\(\pi_{\bullet}:X_{\bullet}\to X\). Here \(X_\bullet\) is a simplicial smooth proper
\(k\)-scheme and the augmentation is a proper hypercover of \(X\), so that proper cohomological descent
applies to the \(\ell\)-adic cohomology used below (for instance, a smooth proper sdh-hypercover as in
Proposition~\ref{prop:smooth-proper-sdh-hypercover}). Set
\(\ol{X}_{\bullet}:=X_{\bullet}\times_{k}\ol{k}\).

By Proposition~\ref{cor:cycle-class-smooth-proper}, the skeletal spectral sequence has only two rows:
\[
  \rH^0_{\Zar}(\ol X_a,\cH^1)=\rH^1_{\et}(\ol X_a,\Ql(1))
  \quad\text{of weight }-1,
\]
and
\[
  \rH^1_{\Zar}(\ol X_a,\cH^1)=\NS(\ol X_a)\otimes\Ql
  \quad\text{of weight }0.
\]

\begin{para}[Notation: \(\NS(\ol{X}_{\bullet})\otimes\Ql\)]\label{para:NS-X-bullet}
We adopt throughout this and the following sections the notation
\[
  \NS(\ol{X}_{\bullet})\otimes\Ql \;:=\; \operatorname{Eq}\Bigl(\NS(\ol{X}_{0})\otimes\Ql\rightrightarrows\NS(\ol{X}_{1})\otimes\Ql\Bigr),
\]
the equalizer of the alternating face-map pullbacks induced by the simplicial structure on
\(\ol{X}_{\bullet}\). This is a Frobenius submodule of \(\NS(\ol{X}_{0})\otimes\Ql\), which is pure of
weight \(0\) by (\ref{para:neron-severi}). Hence \(\NS(\ol{X}_{\bullet})\otimes\Ql\) is also pure of
weight \(0\)
(Lemma~\ref{lem:Frobenius-module}(ii)).
\end{para}

\begin{para}[Simplicial Zariski hypercohomology and \(R^{q}\pi_{\bullet,*}\)]\label{para:simp-setup}
For \(\tau\in\{\Zar,\et,\proet\}\), the simplicial \(\tau\)-topos \((\ol{X}_{\bullet})_{\tau}\) has, as
objects, simplicial sheaves \(\cG_{\bullet}=(\cG_{p})_{p\geq 0}\), where \(\cG_{p}\) is a sheaf on
\((\ol{X}_{p})_{\tau}\) and the simplicial structure is given by face/degeneracy pullback morphisms. The
structure maps \(\lambda_{p}:(\ol{X}_{p})_{\proet}\to(\ol{X}_{p})_{\Zar}\) and
\(\pi_{p}:\ol{X}_{p}\to\ol{X}\) assemble into morphisms of simplicial topoi \(\lambda_{\bullet}\) and
\(\pi_{\bullet}\). Following the conventions of (\ref{para:proet}), we write
\[
  \cH^{q}_{\bullet}(\Ql(1))\;:=\;R^{q}\lambda_{\bullet,*}\Ql(1),
\]
the simplicial Zariski sheaf with degree-\(p\) component \(\cH^{q}_{\ol{X}_{p}}(\Ql(1))\) and
face/degeneracy maps induced by functoriality of \(R^{q}\lambda_{*}\).

Unless another topology is explicitly displayed, \(R^{q}\pi_{\bullet,*}\) in this section means direct
image for the morphism of simplicial Zariski topoi \((\ol{X}_{\bullet})_{\Zar}\to\ol{X}_{\Zar}\).

For a simplicial Zariski sheaf \(\cG_{\bullet}\) on \(\ol{X}_{\bullet}\), hypercohomology is computed by
totalising an injective-resolution double complex in the simplicial and Zariski directions, equivalently
by the skeletal spectral sequence of the filtration by simplicial degree on the totalisation (the
cohomology of the simplicial Zariski topos, \cite[V\(_{\mathrm{bis}}\)]{SGA4_II})
\begin{equation}\label{eq:skeletal-SS}
  E_{1}^{p,q}=\rH^{q}_{\Zar}(\ol{X}_{p},\cG_{p})\;\Longrightarrow\;\bH^{p+q}_{\Zar}(\ol{X}_{\bullet},\cG_{\bullet}),
\end{equation}
with \(d_{1}:E_{1}^{p,q}\to E_{1}^{p+1,q}\) the alternating sum of face-map pullbacks. The same
construction with \(\rH^{q}_{\Zar}(\ol{X}_{p},-)\) replaced by \(R^{q}\pi_{p,*}\) gives a hypercohomology
spectral sequence
\begin{equation}\label{eq:Rqpi-SS}
  E_{1}^{p,q}=R^{q}\pi_{p,*}\cG_{p}\;\Longrightarrow\;R^{p+q}\pi_{\bullet,*}\cG_{\bullet}
\end{equation}
on \(\ol{X}_{\Zar}\). Concretely, \(R^{q}\pi_{\bullet,*}\cG_{\bullet}\) is the Zariski sheafification of
the presheaf
\[
  U\longmapsto\bH^{q}_{\Zar}\bigl(\pi_{\bullet}^{-1}(U),\cG_{\bullet}|_{\pi_{\bullet}^{-1}(U)}\bigr).
\]

The morphism of simplicial topoi
\(\lambda_{\bullet}\colon(\ol{X}_{\bullet})_{\proet}\to(\ol{X}_{\bullet})_{\Zar}\) of (\ref{para:proet})
has an exact pullback \(\lambda_{\bullet}^{*}\) (a left adjoint of \(\lambda_{\bullet,*}\)), so
\(\lambda_{\bullet,*}\) carries injective abelian sheaves on \((\ol{X}_{\bullet})_{\proet}\) to
injectives on \((\ol{X}_{\bullet})_{\Zar}\), in particular to \(\Gamma_{\ol{X}_{\bullet}}\)-acyclic
objects. Grothendieck's spectral sequence for the composition of right derived functors
\cite[Th\'{e}or\`{e}me~2.4.1]{Grothendieck-Tohuku} applied to \(R\Gamma_{(\ol{X}_{\bullet})_{\Zar}}\circ
R\lambda_{\bullet,*}=R\Gamma_{(\ol{X}_{\bullet})_{\proet}}\), evaluated on \(\Ql(1)\), yields the
\emph{simplicial Leray spectral sequence}
\begin{equation}\label{eq:simp-Leray-SS}
  E_{2}^{a,b}\;=\;\rH^{a}_{\Zar}\bigl(\ol{X}_{\bullet},\cH^{b}_{\bullet}(\Ql(1))\bigr)\;\Longrightarrow\;\bH^{a+b}_{\proet}\bigl(\ol{X}_{\bullet},\Ql(1)\bigr)\;=\;\bH^{a+b}_{\et}\bigl(\ol{X}_{\bullet},\Ql(1)\bigr),
\end{equation}
where the abutment identification is (\ref{para:proet}). The spectral sequence is Frobenius-equivariant
because \(\lambda_{\bullet}\) and \(R\Gamma_{(\ol{X}_{\bullet})_{\Zar}}\) commute with the
\(G_{k}\)-action of (\ref{para:galois}).

For the proper hypercoverings used here, the descent input is
\cite[V\(_{\mathrm{bis}}\), Proposition~4.3.2]{SGA4_II}, which establishes that a proper
surjective morphism of schemes is of universal \(F_{n}\)-2-cohomological descent for
\(\bbZ/\ell^{n}\)-coefficients on the \'{e}tale site. The hypercover construction of
\cite[V\(_{\mathrm{bis}}\), Section~5]{SGA4_II} then gives descent for the proper hypercover. Applying this
in each degree to \(\mu_{\ell^{n}}\), compatibly in \(n\), then passing to the limit with the pro-\'etale form
of proper base change for \(\Ql(1)\) \cite[Lemma~6.7.5]{BS15}, gives
\[
  \pi_{\bullet}^{*}:\rH^{n}_{\et}(\ol{X},\Ql(1))\xrightarrow{\;\sim\;}\bH^{n}_{\et}(\ol{X}_{\bullet},\Ql(1)).
\]
This isomorphism is \(G_{k}\)-equivariant. The \(k\)-automorphism
\(\sigma_{\ol{X}_{\bullet}}=\mathrm{id}_{X_{\bullet}}\times\sigma\) acts on \(\pi_{\bullet}\) and on the
displayed map by the same naturality argument used in (\ref{para:cycle-class}). The comparison of
pro-\'etale, continuous \'etale, and \(\ell\)-adic cohomology after \(\otimes\Ql\) (\ref{para:proet})
applies levelwise to each scheme \(\ol{X}_p\), compatibly with the face and degeneracy maps, and so passes
to the skeletal total complexes used here and in \S\ref{sec:auxiliary-sheaf-F}.
\end{para}

\begin{prop}\label{prop:simp-decomp}
There is a Frobenius-equivariant short exact sequence
\begin{equation}\label{eq:simp-weight-ses}
0\to W\to
  \bH^{1}_{\Zar}(\ol{X}_{\bullet},\cH^{1}_{\bullet}(\Ql(1)))
  \to \NS(\ol{X}_{\bullet})\otimes\Ql\to0,
\end{equation}
where \(W\) is pure of weight \(-1\). Consequently the edge map restricts to an isomorphism
\[
  \bH^{1}_{\Zar}(\ol{X}_{\bullet},\cH^{1}_{\bullet}(\Ql(1)))^{w=0}
  \xrightarrow{\ \sim\ }\NS(\ol{X}_{\bullet})\otimes\Ql .
\]
\end{prop}

\begin{proof}
Apply the skeletal spectral sequence \eqref{eq:skeletal-SS} to
\(\cG_\bullet=\cH^1_\bullet(\Ql(1))\):
\[
E_1^{a,b}=\rH^b_{\Zar}(\ol X_a,\cH^1)
\Longrightarrow
\bH^{a+b}_{\Zar}(\ol X_\bullet,\cH^1_\bullet(\Ql(1))).
\]
By Proposition~\ref{cor:cycle-class-smooth-proper}, only the rows \(b=0,1\) occur. The row \(b=1\) is
\(\NS(\ol X_a)\otimes\Ql\), pure of weight \(0\), with \(d_1\) the alternating sum of face-map pullbacks.
Thus
\[
E_2^{0,1}=\operatorname{Eq}\bigl(\NS(\ol X_0)\otimes\Ql\rightrightarrows
\NS(\ol X_1)\otimes\Ql\bigr)=\NS(\ol X_\bullet)\otimes\Ql .
\]
The row \(b=0\) is \(\rH^1_{\et}(\ol X_a,\Ql(1))\), pure of weight \(-1\), so every subquotient
\(E_r^{a,0}\) is pure of weight \(-1\). The only possible higher differential affecting total degree
\(1\) is
\[
d_2:E_2^{0,1}\longrightarrow E_2^{2,0},
\]
and it is zero because there are no non-zero maps from weight \(0\) to weight \(-1\). The five-term exact
sequence therefore gives
\[
0\to E_2^{1,0}\to
\bH^1_{\Zar}(\ol X_\bullet,\cH^1_\bullet(\Ql(1)))
\to E_2^{0,1}\to0.
\]
Set \(W:=E_2^{1,0}\). The last display is \eqref{eq:simp-weight-ses}. The weight-zero conclusion follows
by applying the exact functor \(V\mapsto V^{w=0}\).
\end{proof}

\begin{prop}\label{prop:mu}
The levelwise cycle classes induce a Frobenius-equivariant map
\[
  \mu:\NS(\ol X_\bullet)\otimes\Ql\longrightarrow
  \bH^1_{\Zar}(\ol X_\bullet,\cH^1_\bullet(\Ql(1))).
\]
It is injective, identifies its image with the weight-\(0\) subspace, and has cokernel pure of weight
\(-1\).
\end{prop}

\begin{proof}
The \(d\log_\ell\) map of (\ref{para:cycle-class}) induces a morphism of simplicial Zariski sheaves
\(\cO^\times_{\ol X_\bullet}\otimes\Ql\to\cH^1_\bullet(\Ql(1))\). Let \(\mu\) be the map it induces on
hypercohomology,
\[
  \mu\colon\bH^1_{\Zar}(\ol X_\bullet,\cO^\times_{\ol X_\bullet}\otimes\Ql)\longrightarrow\bH^1_{\Zar}(\ol X_\bullet,\cH^1_\bullet(\Ql(1))).
\]
By the functoriality and \(G_k\)-equivariance of \(d\log_\ell\) (\ref{para:cycle-class}), \(\mu\) is
Frobenius-equivariant. Its source is \(\NS(\ol X_\bullet)\otimes\Ql\). In the skeletal spectral sequence \eqref{eq:skeletal-SS}
for \(\cO^\times_{\ol X_\bullet}\otimes\Ql\) the row \(b=0\) vanishes (global units on the smooth proper
levels are finite products of \(\ol k^{\times}\), which vanish after \(\otimes\Ql\) by Lemma~\ref{lem:torsion-ag}), the row \(b=1\) is \(\Pic(\ol X_a)\otimes\Ql=\NS(\ol
X_a)\otimes\Ql\) (Lemma~\ref{lem:pic-via-cartier} and (\ref{para:neron-severi})), and \(b\geq2\) vanishes
(Lemma~\ref{lem:pic-via-cartier}). Hence \(\bH^1=E_2^{0,1}=\NS(\ol X_\bullet)\otimes\Ql\)
(\ref{para:NS-X-bullet}).

Composed with the edge map of \eqref{eq:simp-weight-ses}, \(\mu\) is levelwise the cycle-class
isomorphism \(\NS(\ol X_a)\otimes\Ql\xrightarrow{\sim}\rH^1_{\Zar}(\ol X_a,\cH^1)\) of
Proposition~\ref{cor:cycle-class-smooth-proper}, so on the equalizers it is the identity of
\(E_2^{0,1}=\NS(\ol X_\bullet)\otimes\Ql\).
Thus \(\mu\) is a section of that edge map, in particular injective. Its source is pure of weight \(0\),
so \(\operatorname{im}\mu\) lies in the weight-\(0\) part of the target, and as the kernel \(W\) of
\eqref{eq:simp-weight-ses} is pure of weight \(-1\), the exact functor \(V\mapsto V^{w=0}\) makes the edge
map restrict to an isomorphism from that weight-\(0\) part onto \(E_2^{0,1}\). Hence
\(\operatorname{im}\mu\) is exactly the weight-\(0\) part, and \(\coker\mu\cong W\) is pure of weight
\(-1\).
\end{proof}

\section{The auxiliary sheaf \texorpdfstring{\(\cF\)}{F}}
\label{sec:auxiliary-sheaf-F}

Retain the standing setup of \S\ref{sec:simp-cycle-class}. Here \(X\) is a proper \(k\)-scheme,
\(\pi_\bullet:X_\bullet\to X\) a smooth proper hypercover in the sense fixed there, and
\(\ol{X}_\bullet:=X_\bullet\times_k\ol{k}\).

\begin{lem}\label{lem:F2-weight}
In the simplicial Leray spectral sequence for
\(\lambda_\bullet:(\ol{X}_\bullet)_\proet\to(\ol{X}_\bullet)_\Zar\),
\[
  E_2^{a,b}=\rH^a_\Zar(\ol{X}_\bullet,\cH^b_\bullet(\Ql(1)))\Longrightarrow
  \bH^{a+b}_\et(\ol{X}_\bullet,\Ql(1)),
\]
the filtration piece \(F^2\bH^2_\et(\ol{X}_\bullet,\Ql(1))=E_\infty^{2,0}\) is pure of weight \(-2\).
\end{lem}

\begin{proof}
Since \(F^3=E_\infty^{3,-1}=0\), \(F^2=E_\infty^{2,0}\), a subquotient of
\[
  E_2^{2,0}\;=\;\rH^2_\Zar(\ol{X}_\bullet,\cH^0_\bullet(\Ql(1)))\;=\;\rH^2_\Zar(\ol{X}_\bullet,\Ql(1)),
\]
where the second equality uses \(\cH^0_{\ol{X}_p}(\Ql(1))=R^0\lambda_{\ol{X}_p,*}\Ql(1)=\Ql(1)\), the
locally constant sheaf of value \(\Ql(1)\) on each connected component of \(\ol{X}_p\) (proof of
Lemma~\ref{lem:h1-proet-presheaf-is-sheaf}, applied componentwise, with \(G_{\ol{k}}\) trivial).

Compute \(\rH^2_\Zar(\ol{X}_\bullet,\Ql(1))\) via the skeletal spectral sequence \eqref{eq:skeletal-SS}
with \(\cG_\bullet=\Ql(1)\), so \(E_1^{p,q}=\rH^q_\Zar(\ol{X}_p,\Ql(1))\). Each connected component of
\(\ol{X}_p\) is irreducible noetherian, so \(\Ql(1)\) is flasque componentwise. Hence
\(\rH^q_\Zar(\ol{X}_p,\Ql(1))=0\) for \(q\geq 1\), and only the \(q=0\) row survives:
\[
  E_1^{p,0}\;=\;\rH^0_\Zar(\ol{X}_p,\Ql(1))\;=\;\Ql(1)^{\pi_0(\ol{X}_p)}.
\]
The geometric Frobenius permutes the finitely many connected components of \(\ol{X}_p\) through a finite
quotient, and acts on \(\Ql(1)\) by weight \(-2\). Hence each \(E_1^{p,0}\) is a Frobenius module pure of
weight \(-2\), and so is every cohomology subquotient of the complex \(C^\bullet=E_1^{\bullet,0}\)
(Lemma~\ref{lem:Frobenius-module}(ii)). In particular
\(\rH^2_\Zar(\ol{X}_\bullet,\Ql(1))=\rH^2(C^\bullet)\) is pure of weight \(-2\), and so is its
subquotient \(E_\infty^{2,0}\).
\end{proof}

\begin{para}[The auxiliary sheaf \(\cF\) and the augmentation \(u\)]\label{para:F-and-u}
Following the conventions of (\ref{para:simp-setup}), define
\[
  \cF\;:=\;R^0\pi_{\bullet,*}\cH^1_\bullet(\Ql(1))
  \;=\;\operatorname{Eq}\Bigl(\pi_{0,*}\cH^1_0(\Ql(1))\rightrightarrows \pi_{1,*}\cH^1_1(\Ql(1))\Bigr),
\]
a Zariski sheaf on \(\ol{X}\). The simplicial morphism \(\pi_\bullet:\ol{X}_\bullet\to\ol{X}\) furnishes
a base-change map \(\beta:\pi_\bullet^*\cH^1_{\ol{X}}(\Ql(1))\to \cH^1_\bullet(\Ql(1))\) of simplicial
Zariski sheaves on \(\ol{X}_\bullet\). Its adjoint under \((\pi_\bullet^*,\pi_{\bullet,*})\) is a
morphism of Zariski sheaves on \(\ol{X}\),
\[
  u:\;\cH^1_{\ol{X}}(\Ql(1))\;\longrightarrow\;\cF,
\]
the \emph{augmentation}. Since \(\pi_\bullet\) and \(\lambda_\bullet\) are \(G_k\)-equivariant
(\ref{para:simp-setup}), \(\cF\) is a \(G_k\)-equivariant Zariski sheaf, so \(\rH^1_\Zar(\ol{X},\cF)\)
carries a Frobenius action. The injective Leray edge
\(j_1\colon\rH^1_\Zar(\ol{X},\cF)\hookrightarrow\bH^1_\Zar(\ol{X}_\bullet,\cH^1_\bullet(\Ql(1)))\) for
\(\pi_\bullet\) (five-term sequence) embeds it into the Frobenius module of
Proposition~\ref{prop:simp-decomp}. Thus \(\rH^1_\Zar(\ol{X},\cF)\) is a Frobenius module, with weight
grading induced by \(j_1\).
\end{para}

\begin{prop}\label{prop:u-inj}
Assume \(X\) is a proper scheme over \(k\), and let
\(\pi_\bullet:X_\bullet\to X\) be a smooth proper hypercover in the sense of
\S\ref{sec:simp-cycle-class}. With \(\cF\) and \(u\) defined by
(\ref{para:F-and-u}), the map induced by \(u\) on cohomology
\[
  u_*:\;\rH^1_{\Zar}(\ol{X},\cH^1)\;\longrightarrow\;\rH^1_\Zar(\ol{X},\cF)
\]
restricts to an injection on weight-zero subspaces:
\[
  (u_*)^{w=0}:\;\rH^1_{\Zar}(\ol{X}, \cH^1)^{w=0}\;\lhook\joinrel\longrightarrow\;\rH^1_\Zar(\ol{X},\cF)^{w=0}.
\]
\end{prop}

\begin{proof}
Let \(\alpha\in\rH^1_{\Zar}(\ol{X}, \cH^1)^{w=0}=E_2^{1,1}(\ol{X})^{w=0}\) in the Leray spectral sequence for
\(\lambda_{\ol{X}}\colon\ol{X}_\proet\to\ol{X}_\Zar\), and assume \(u_*(\alpha)=0\).

The augmentation is adjoint to the base-change map
\(\beta:\pi_\bullet^*\cH^1_{\ol X}(\Ql(1))\to\cH^1_\bullet(\Ql(1))\). Functoriality of the Leray edge for
the adjunction \((\pi_\bullet^*,\pi_{\bullet,*})\) gives
\begin{equation}\label{eq:u-adj}
  j_1\circ u_*\;=\;\beta_*\circ\pi_\bullet^*:\;\rH^1_\Zar(\ol{X},\cH^1)\to \bH^1_\Zar(\ol{X}_\bullet,\cH^1_\bullet(\Ql(1))),
\end{equation}
where \(j_1\) is the Leray edge for \(\pi_\bullet\), injective by the five-term sequence. Thus
\[
  \eta:=\beta_*\pi_\bullet^*\alpha\in E_2^{1,1}(\ol X_\bullet)
\]
is zero.

The Zariski sheaf \(R^{0}\lambda_{\ol{X},*}\Ql(1)\) on \(\ol{X}_{\Zar}\) is the constant sheaf
\(\Ql(1)\). On any open \(U\subseteq\ol{X}\), its sections are
\(\rH^{0}_{\proet}(U,\Ql(1))=\Ql(1)^{\pi_{0}(U)}\), which is the section group of the constant Zariski
sheaf \(\Ql(1)\). No connectedness of \(\ol{X}\) is needed. By Lemma~\ref{lem:Zar-Ql-weight} applied to
\(T=X\), each
\[
  E_2^{p,0}(\ol{X})\;=\;\rH^{p}_{\Zar}(\ol{X},R^{0}\lambda_{\ol{X},*}\Ql(1))\;=\;\rH^{p}_{\Zar}(\ol{X},\Ql(1))
\]
is a Frobenius module pure of weight \(-2\). (Note that for the singular reducible \(\ol{X}\) the simple
flasqueness argument from Lemma~\ref{lem:F2-weight} on the simplicial side does \emph{not} apply
directly. Lemma~\ref{lem:Zar-Ql-weight} treats this via induction on the number of irreducible
components.)

The outgoing differential \(d_2\colon E_2^{1,1}(\ol{X})\to E_2^{3,0}(\ol{X})\) is Frobenius-equivariant.
Its restriction to the weight-\(0\) subspace \(E_2^{1,1}(\ol{X})^{w=0}\) takes values in the weight-\(0\)
part of \(E_2^{3,0}(\ol{X})\), but \(E_2^{3,0}(\ol{X})\) is pure of weight \(-2\), so its weight-\(0\)
part vanishes. Therefore \(d_2\alpha=0\). For \(r\geq 3\), the
differential \(d_r:E_r^{1-r,r}\to E_r^{1,1}\) has source first index \(<0\) hence source zero, and
\(d_r:E_r^{1,1}\to E_r^{1+r,2-r}\) has target second index \(<0\) hence target zero. Therefore \(\alpha\)
lies in
\[
  E_\infty^{1,1}(\ol{X})\;=\;\ker\bigl(d_2\colon E_2^{1,1}(\ol{X})\to
  E_2^{3,0}(\ol{X})\bigr)\;\subseteq\;E_2^{1,1}(\ol{X}).
\]
The same indexing argument applied on the simplicial side gives
\[
  E_\infty^{1,1}(\ol{X}_\bullet)
  =
  \ker\bigl(d_2\colon E_2^{1,1}(\ol{X}_\bullet)\to E_2^{3,0}(\ol{X}_\bullet)\bigr).
\]

The square of topoi
\[
\begin{tikzcd}[column sep=large, row sep=large]
(\ol{X}_\bullet)_\proet \arrow[r, "\pi_\bullet"] \arrow[d, "\lambda_\bullet"'] & \ol{X}_\proet \arrow[d, "\lambda_{\ol{X}}"] \\
(\ol{X}_\bullet)_\Zar \arrow[r, "\pi_\bullet"] & \ol{X}_\Zar
\end{tikzcd}
\]
induces a filtered morphism from the Leray spectral sequence \eqref{eq:base-leray-SS} for
\(\lambda_{\ol{X}}\) to the simplicial Leray spectral sequence \eqref{eq:simp-Leray-SS}. On
\(E_2^{a,b}\) it is \(\beta_*\pi_\bullet^*\), and on abutments it is the cohomological-descent isomorphism
\[
  \pi_{\bullet,\et}^*:\rH^2_\et(\ol X,\Ql(1))\xrightarrow{\sim}
  \bH^2_\et(\ol X_\bullet,\Ql(1)).
\]
Concretely, both spectral sequences arise from the canonical (Postnikov) filtration on the derived
pushforward along \(\lambda\). Here \eqref{eq:base-leray-SS} comes from the filtered complex
\(R\Gamma_{\Zar}(\ol X,R\lambda_{\ol X,*}\Ql(1))\), and \eqref{eq:simp-Leray-SS} from
\(R\Gamma_{\Zar}(\ol X_\bullet,R\lambda_{\bullet,*}\Ql(1))\). The base-change map \(\beta\) and the
adjunction \((\pi_\bullet^{*},\pi_{\bullet,*})\) furnish a morphism of filtered complexes
\[
  R\Gamma_{\Zar}(\ol X,R\lambda_{\ol X,*}\Ql(1))\longrightarrow
  R\Gamma_{\Zar}(\ol X_\bullet,R\lambda_{\bullet,*}\Ql(1))
\]
respecting the canonical filtrations. The induced morphism of spectral sequences is
\(\beta_*\pi_\bullet^{*}\) on \(E_2\) and \(\pi_{\bullet,\et}^{*}\) on abutments, and it respects the
abutment filtrations and edge maps because these are read off from the filtered complex.

Let \(\bar\alpha\in E_\infty^{1,1}(\ol{X})=F^1\rH^2_\et(\ol{X},\Ql(1))/F^2\rH^2_\et(\ol{X},\Ql(1))\) be
the class represented by \(\alpha\). Since
\(F^2\rH^2_\et(\ol{X},\Ql(1))=E_\infty^{2,0}(\ol{X})\) is a subquotient of the pure weight-\((-2)\)
module \(E_2^{2,0}(\ol{X})\), it is pure of weight \(-2\). Applying the exact functor \(V\mapsto
V^{w=0}\) to
\[
0\to F^2\rH^2_\et(\ol{X},\Ql(1))\to F^1\rH^2_\et(\ol{X},\Ql(1))\to E_\infty^{1,1}(\ol{X})\to0
\]
therefore gives an isomorphism
\[
  \bigl(F^1\rH^2_\et(\ol{X},\Ql(1))\bigr)^{w=0}
  \xrightarrow{\;\sim\;}
  E_\infty^{1,1}(\ol{X})^{w=0}.
\]
Choose the unique weight-zero lift \(\widetilde{\alpha}\in F^1\rH^2_\et(\ol{X},\Ql(1))^{w=0}\) mapping to
\(\bar\alpha\).

The morphism of spectral sequences sends \(\bar\alpha\) on the \(E_\infty\)-page to the image of
\(\eta\in E_2^{1,1}(\ol{X}_\bullet)\) in \(E_\infty^{1,1}(\ol{X}_\bullet)\). Since \(\eta=0\),
the \(E_\infty^{1,1}\)-graded component of \(\pi_{\bullet,\et}^{*}\widetilde{\alpha}\) is zero.
Since the pullback map is filtered, this gives
\[
  \pi_{\bullet,\et}^{*}\widetilde{\alpha}\in F^2\bH^2_\et(\ol{X}_\bullet,\Ql(1)).
\]
The pullback is Frobenius-equivariant, so \(\pi_{\bullet,\et}^{*}\widetilde{\alpha}\) still has weight
\(0\). Lemma~\ref{lem:F2-weight} says that the target \(F^2\bH^2_\et(\ol{X}_\bullet,\Ql(1))\) is pure of
weight \(-2\), hence has no weight-zero subspace. Thus \(\pi_{\bullet,\et}^{*}\widetilde{\alpha}=0\).
Since cohomological descent makes \(\pi_{\bullet,\et}^{*}\) an isomorphism on abutments,
\(\widetilde{\alpha}=0\), and therefore \(\bar\alpha=0\) in \(E_\infty^{1,1}(\ol{X})\).

Finally, on the weight-zero part of \(E_2^{1,1}(\ol{X})\), the inclusion
\(E_\infty^{1,1}(\ol{X})^{w=0}\subseteq E_2^{1,1}(\ol{X})^{w=0}\) is an \emph{equality}. The kernel of
\(d_2\) contains every weight-zero element of \(E_2^{1,1}(\ol{X})\), and incoming differentials at
\((1,1)\) vanish by indexing. Therefore \(\alpha\) viewed in
\(\rH^1_{\Zar}(\ol{X}, \cH^1)^{w=0}=E_2^{1,1}(\ol{X})^{w=0}\) agrees with \(\alpha\) viewed in
\(E_\infty^{1,1}(\ol{X})\), and we conclude \(\alpha=0\).
\end{proof}

\section{Proof of the main theorem}
\label{sec:main-theorem-proof}

The proof assembles the outcomes of the preceding sections. The smooth-proper computation (\S\ref{sec:dlog-smooth-proper}),
transported along a smooth proper hypercover, feeds the simplicial weight decomposition
(\S\ref{sec:simp-cycle-class}) and the auxiliary sheaf \(\cF\) (\S\ref{sec:auxiliary-sheaf-F}). We first pass
to the seminormalisation, then prove the seminormal case by comparing rational Picard descent with the
weight-zero part through \(\cF\).

\medskip\noindent\textbf{Reduction to the seminormal case.}\par\smallskip

\begin{lem}\label{lem:Q-pN-torsion}
Let \(K\) be a perfect field of characteristic \(p>0\), \(X\) a reduced \(K\)-scheme of finite type, and
\(\rho\colon X^{\mathrm{sn}}\to X\) the seminormalisation. The cokernel
\[
  \cQ\;:=\;\rho_{*}\cO_{X^{\mathrm{sn}}}^{\times}/\cO_{X}^{\times}
\]
of Zariski sheaves of abelian groups on \(X\) is annihilated by some power of \(p\).
\end{lem}

\begin{proof}
We first establish the additive analogue. There exists \(N\geq 0\) such that
\begin{equation}\label{eq:additive-frob-factor}
  (\rho_{*}\cO_{X^{\mathrm{sn}}})^{p^{N}}\;\subseteq\;\cO_{X}\quad\text{as subsheaves of }\rho_{*}\cO_{X^{\mathrm{sn}}}.
\end{equation}
The statement is local on \(X\), so by quasi-compactness we may pass to a finite affine cover and bound
\(N\) on each member. Replace \(X\) by \(\Spec A\) and \(X^{\mathrm{sn}}\) by \(\Spec B\). The inclusion
\(A\hookrightarrow B=A^{\mathrm{sn}}\) is finite \cite[\S 1]{GrecoTraverso80}.

Let \(A_{0}:=A\) and inductively \(A_{i+1}\) the \(A_{i}\)-subalgebra of \(Q(A)\) generated by \(\{c\in
Q(A):c^{2},c^{3}\in A_{i}\}\). By the explicit construction of the seminormalisation \cite[Tag~0EUK]{stacks-project}, \(B=\bigcup_{i\geq
0}A_{i}\). Since \(B\) is finite over \(A\), the chain stabilises. Let \(N\) be minimal with \(A_{N}=B\).

The absolute Frobenius \(F\colon B\to B\) sends \(A_{i}\) into \(A_{i-1}\) for every \(i\geq 1\). Every
\(x\in A_{i}\) is an \(A_{i-1}\)-polynomial in finitely many \(c_{j}\in Q(A)\) with
\(c_{j}^{2},c_{j}^{3}\in A_{i-1}\). Writing \(p=2u+3v\) with \(u,v\in\bbZ_{\geq 0}\) gives
\(c_{j}^{p}=(c_{j}^{2})^{u}(c_{j}^{3})^{v}\in A_{i-1}\), and Frobenius is additive in characteristic
\(p\), so \(x^{p}\in A_{i-1}\). Iterating, \(F^{N}(B)\subseteq A\), which is
\eqref{eq:additive-frob-factor}.

For the multiplicative statement, let \(U\subseteq X\) be a Zariski open and
\(b\in\rho_{*}\cO_{X^{\mathrm{sn}}}^{\times}(U)\). By~\eqref{eq:additive-frob-factor} we have
\(b^{p^{N}}\in\cO_{X}(U)\). Since \(b^{p^{N}}\) is a unit in \(\rho_{*}\cO_{X^{\mathrm{sn}}}(U)\) and
\(\rho\) is surjective, it is non-vanishing on \(U\), so \(b^{p^{N}}\in\cO_{X}^{\times}(U)\). Therefore
\(p^{N}\cdot[b]=0\) in \(\cQ(U)\).
\end{proof}

\begin{prop}\label{prop:pic-rationalisation}
Let \(K\) be a perfect field of characteristic \(p>0\), and \(X\) a reduced \(K\)-scheme of finite type.
The pullback along the seminormalisation \(\rho\colon X^{\mathrm{sn}}\to X\) induces an isomorphism
\[
  \rho^{*}\colon\Pic(X)\otimes_{\bbZ}\Ql\;\xrightarrow{\;\sim\;}\Pic(X^{\mathrm{sn}})\otimes_{\bbZ}\Ql.
\]
\end{prop}

\begin{proof}
The morphism \(\rho\) is a Zariski homeomorphism of underlying topological spaces, so for every sheaf
\(\cF\) on \(X^{\mathrm{sn}}_{\Zar}\), \(\rH^{i}(X,\rho_{*}\cF)=\rH^{i}(X^{\mathrm{sn}},\cF)\) for all
\(i\). In particular
\begin{equation}\label{eq:pushforward-pic}
  \rH^{1}(X,\rho_{*}\cO_{X^{\mathrm{sn}}}^{\times})\;=\;\Pic(X^{\mathrm{sn}}).
\end{equation}
The canonical map \(\cO_{X}^{\times}\hookrightarrow\rho_{*}\cO_{X^{\mathrm{sn}}}^{\times}\) is injective,
so by Lemma~\ref{lem:Q-pN-torsion} we have a short exact sequence of Zariski sheaves on \(X\),
\begin{equation}\label{eq:Q-SES}
  1\;\longrightarrow\;\cO_{X}^{\times}\;\longrightarrow\;\rho_{*}\cO_{X^{\mathrm{sn}}}^{\times}\;\longrightarrow\;\cQ\;\longrightarrow\;1,
\end{equation}
with \(\cQ\) annihilated by \(p^{N}\). Taking Zariski cohomology and using~\eqref{eq:pushforward-pic}
gives the long exact sequence
\[
  \rH^{0}(X,\cQ)\longrightarrow\Pic(X)\xrightarrow{\;\rho^{*}\;}\Pic(X^{\mathrm{sn}})\longrightarrow\rH^{1}(X,\cQ).
\]
Both end terms are \(p^{N}\)-torsion. Tensoring with \(\Ql\) (\(\ell\neq p\)) kills them and gives the
isomorphism.
\end{proof}

\begin{cor}\label{cor:NS-rationalisation}
For a proper reduced \(\ol{k}\)-scheme \(X\) with seminormalisation \(\rho\colon X^{\mathrm{sn}}\to X\),
the pullback \(\rho^{*}\colon\NS(X)\otimes\Ql\to\NS(X^{\mathrm{sn}})\otimes\Ql\) is an isomorphism.
\end{cor}

\begin{proof}
Apply (\ref{para:neron-severi}) on each side (\(\Pic^{0}\otimes\Ql=0\)) and
Proposition~\ref{prop:pic-rationalisation}.
\end{proof}

\begin{prop}\label{prop:reduction-to-seminormal}
Suppose Theorem~\ref{thm:main-result} holds for proper seminormal reduced \(k\)-schemes. Then it holds
for all proper reduced \(k\)-schemes.
\end{prop}

\begin{proof}
Let \(X\) be proper reduced over \(k\), \(\rho_X\colon X^{\mathrm{sn}}\to X\) its seminormalisation.
Since \(k\) is perfect and \(X\) is reduced, \(X^{\mathrm{sn}}\) is a proper reduced seminormal
\(k\)-scheme \cite[\S 5]{GrecoTraverso80}, and the seminormalisation commutes with base change to
\(\ol{k}\) \cite[Corollary~5.7]{GrecoTraverso80}, so \(\ol{X^{\mathrm{sn}}}=(\ol{X})^{\mathrm{sn}}\) and
\(\rho_{\ol{X}}\colon\ol{X^{\mathrm{sn}}}\to\ol{X}\) is the seminormalisation of \(\ol{X}\). The
hypothesis applied to \(X^{\mathrm{sn}}\) gives a \(G_{k}\)-equivariant isomorphism
\[
  \widetilde{c}_{1}^{\,\mathrm{sn}}\colon\NS(\ol{X^{\mathrm{sn}}})\otimes\Ql\;\xrightarrow{\;\sim\;}\;\rH^1_{\Zar}(\ol{X^{\mathrm{sn}}}, \cH^1)^{w=0}.
\]

\emph{NS-side.} By Corollary~\ref{cor:NS-rationalisation} applied over \(\ol{k}\), the pullback
\(\rho_{\ol{X}}^{*}\colon\NS(\ol{X})\otimes\Ql\xrightarrow{\sim}\NS(\ol{X^{\mathrm{sn}}})\otimes\Ql\) is
a \(G_{k}\)-equivariant isomorphism.

\emph{ZL-side.} The universal homeomorphism \(\rho_{\ol{X}}\) induces an equivalence of small étale topoi
\cite[Exp.~VIII, Th\'{e}or\`{e}me~1.1]{ArtinGrothendieckSGA4} and a homeomorphism on underlying Zariski
spaces. The étale-topos equivalence gives isomorphisms
\(\rho_{\ol{X}}^{*}\colon\rH^{q}_{\et}(V,\mu_{\ell^{n}})\xrightarrow{\sim}\rH^{q}_{\et}(\rho_{\ol{X}}^{-1}(V),\mu_{\ell^{n}})\)
for every Zariski open \(V\subseteq\ol{X}\) and every \(q,n\). Passing to \(\varprojlim_{n}\),
identifying with pro-étale cohomology via (\ref{para:proet}), and tensoring with \(\Ql\) gives the same
for \(\Ql(1)\)-coefficients. Sheafifying along the Zariski homeomorphism identifies
\(\cH^{1}_{\ol{X}}(\Ql(1))\) with \(\cH^{1}_{\ol{X^{\mathrm{sn}}}}(\Ql(1))\) as Zariski sheaves on the
common underlying topological space, and consequently
\(\rH^1_{\Zar}(\ol{X^{\mathrm{sn}}}, \cH^1)=\rH^1_{\Zar}(\ol{X}, \cH^1)\) as \(G_{k}\)-modules with weight grading.

\emph{Composition.} Combining,
\[
  \widetilde{c}_{1}\;=\;(\rho_{\ol{X}}^{*})^{-1}\circ\widetilde{c}_{1}^{\,\mathrm{sn}}\circ\rho_{\ol{X}}^{*}\colon\NS(\ol{X})\otimes\Ql\;\xrightarrow{\;\sim\;}\;\rH^1_{\Zar}(\ol{X}, \cH^1)^{w=0},
\]
where the outer factors are isomorphisms by the NS-side and ZL-side identifications above and the middle
factor by hypothesis. The compatibility of \(\widetilde{c}_{1}\) and
\(\widetilde{c}_{1}^{\,\mathrm{sn}}\) under \(\rho_{\ol{X}}^{*}\) is the contravariant functoriality
of (\ref{para:cycle-class}).
\end{proof}

\medskip\noindent\textbf{Rational N\'eron--Severi descent.}\par\smallskip

\begin{lem}\label{lem:pic-descent-Fpbar}
Let \(X\) be a proper reduced \(k\)-scheme, and let \(\pi_\bullet:X_\bullet\to X\) be a proper
\(h\)-hypercover with each \(X_p\) reduced. Put \(\ol X_\bullet=X_\bullet\times_k\ol k\). Then pullback
induces an isomorphism
\[
  \pi_\bullet^{*}:\Pic(\ol X)\otimes\Ql\xrightarrow{\ \sim\ }
  \operatorname{Eq}\bigl(\Pic(\ol X_0)\otimes\Ql\rightrightarrows\Pic(\ol X_1)\otimes\Ql\bigr).
\]
\end{lem}

\begin{proof}
Write \(Y=\ol X\), \(Y_\bullet=\ol X_\bullet\), \(Z=Y^{\mathrm{perf}}\), and
\(Z_\bullet=Y_\bullet^{\mathrm{perf}}\). Perfection commutes with finite limits. Therefore the matching
maps of \(Z_\bullet\to Z\) are the perfections of the matching maps of \(Y_\bullet\to Y\).

We first record the descent input. An \(h\)-cover of qcqs schemes in characteristic \(p\) is a \(v\)-cover
\cite[Example~2.3]{BhattScholze17}. Since maps from perfect valuation rings factor through perfection,
passing to perfections does not change the induced \(v\)-covering condition. Hence
\(Z_\bullet\to Z\) is a \(v\)-hypercover in the category of perfect qcqs schemes.

For a perfect scheme \(T\), Bhatt--Scholze prove that vector bundles on \(W_n(T)\) form a \(v\)-stack
\cite[Theorem~4.1(ii)]{BhattScholze17}. Taking \(n=1\) gives \(v\)-descent for vector bundles on \(T\). Since line bundles form the
rank-one substack and rank is \(v\)-local, the Picard groupoid is likewise a \(v\)-stack. Thus
\[
  \rH^0_v(T,\cO_T^\times)=\Gamma(T,\cO_T^\times),
  \qquad
  \rH^1_v(T,\cO_T^\times)=\Pic(T).
\]
The hypercover spectral sequence
\[
  E_1^{a,b}=\rH^b_v(Z_a,\cO^\times)\Longrightarrow \rH^{a+b}_v(Z,\cO^\times)
\]
\cite[V, (7.4.0.3)]{SGA4_II} therefore gives the exact sequence
\[
  0\to \check H^1(Z_\bullet,\cO^\times)\to \Pic(Z)\to
  \operatorname{Eq}\bigl(\Pic(Z_0)\rightrightarrows\Pic(Z_1)\bigr)
  \to \check H^2(Z_\bullet,\cO^\times),
\]
where the \v{C}ech groups are computed from the cosimplicial groups
\(\Gamma(Z_a,\cO_{Z_a}^\times)\).

The two \v{C}ech terms are torsion. Indeed, each \(Y_a\) is proper and reduced over \(\ol k\), so
\(\Gamma(Y_a,\cO_{Y_a}^\times)\) is a finite product of copies of \(\ol k^\times\), hence torsion. Moreover
\[
  \Gamma(Z_a,\cO_{Z_a}^\times)=\Gamma(Y_a,\cO_{Y_a}^\times)[1/p],
\]
because perfection is the filtered colimit along Frobenius and Frobenius acts on units by
\(u\mapsto u^p\). Thus \(\check H^1\) and \(\check H^2\) are torsion, and they vanish after tensoring with
\(\Ql\).

Finally, \(\Pic(Y)[1/p]\cong\Pic(Y^{\mathrm{perf}})=\Pic(Z)\), functorially in \(Y\), by
\cite[Lemma~3.5]{BhattScholze17}. The same holds levelwise for \(Y_\bullet\). Since \(\ell\neq p\),
tensoring the preceding exact sequence with \(\Ql\) gives the displayed isomorphism for
\(Y_\bullet\to Y\).
\end{proof}

\begin{prop}\label{prop:NS-descent-Fpbar}
Let \(X\) and \(\pi_\bullet:X_\bullet\to X\) satisfy the hypotheses of
Lemma~\ref{lem:pic-descent-Fpbar}. Then
\[
  \pi_\bullet^{*}\colon\NS(\ol X)\otimes\Ql
  \xrightarrow{\;\sim\;}
  \operatorname{Eq}\bigl(\NS(\ol X_0)\otimes\Ql\rightrightarrows\NS(\ol X_1)\otimes\Ql\bigr)
\]
is a \(G_k\)-equivariant isomorphism. The hypercover \(X_\bullet\to X\) is defined over \(k\), so
\(\pi_\bullet^{*}\) and the face-map pullbacks commute with the \(G_k\)-action.
\end{prop}

\begin{proof}
By (\ref{para:neron-severi}), \(\Pic^{0}(Y)\otimes\Ql=0\) for any proper
\(Y/\ol k\) (Lemma~\ref{lem:torsion-ag}), so \(\Pic(Y)\otimes\Ql=\NS(Y)\otimes\Ql\) for
\(Y=\ol X\) and componentwise for each \(\ol X_p\). Flatness of \(\Ql\) over \(\bbZ\) makes
the equalizer commute with \(\otimes\Ql\), so the equalizers of \(\NS\otimes\Ql\) and
\(\Pic\otimes\Ql\) along the face maps agree. By (\ref{para:NS-X-bullet}) the
former is \(\NS(\ol{X}_\bullet)\otimes\Ql\). Lemma~\ref{lem:pic-descent-Fpbar} identifies
it with \(\Pic(\ol X)\otimes\Ql=\NS(\ol X)\otimes\Ql\), giving the displayed isomorphism.
\end{proof}

\medskip\noindent\textbf{Proof of the seminormal case.}\par\smallskip

For the rest of this section, assume \(X\) is a proper seminormal reduced \(k\)-scheme. Let
\(\pi_\bullet:X_\bullet\to X\) be a smooth proper sdh-hypercover
(Proposition~\ref{prop:smooth-proper-sdh-hypercover}), and set
\(\ol{X}_\bullet=X_\bullet\times_k\ol{k}\).

\begin{prop}\label{prop:delta-cokernel}
There is a natural \(G_k\)-equivariant map
\[
  \delta:\NS(\ol X)\otimes\Ql\longrightarrow \rH^1_\Zar(\ol X,\cF)
\]
such that
\[
  \delta=u_*\circ\widetilde{c}_{1}.
\]
Moreover \(\delta\) induces an isomorphism
\[
  \delta:\NS(\ol X)\otimes\Ql\xrightarrow{\;\sim\;}\rH^1_\Zar(\ol X,\cF)^{w=0}.
\]
\end{prop}

\begin{proof}
\(\ol X\) is seminormal because seminormality is preserved by separable extension
\cite[Corollary~5.7]{GrecoTraverso80}. Hence Lemma~\ref{lem:sdh-degree-zero}, applied over \(\ol k\),
gives
\[
  R^0\pi_{\bullet,*}\bigl(\cO^\times_{\ol X_\bullet}\otimes\Ql\bigr)
  =\cO^\times_{\ol X}\otimes\Ql ,
  \qquad
  R^0\pi_{\bullet,*}\cH^1_\bullet(\Ql(1))=\cF
\]
by definition of \(\cF\). The two Leray five-term sequences for \(\pi_\bullet\), with coefficients
\(\cO^\times_{\ol X_\bullet}\otimes\Ql\) and \(\cH^1_\bullet(\Ql(1))\), are linked by
\(d\log_\ell\). In degree \(1\), the top row identifies with
\[
  \NS(\ol X)\otimes\Ql
  \xrightarrow{\;\pi_\bullet^*\;}
  \NS(\ol X_\bullet)\otimes\Ql .
\]
On \(\ol X\) this is Lemma~\ref{lem:pic-via-cartier} and (\ref{para:neron-severi}). On
\(\ol X_\bullet\), the skeletal spectral sequence for
\(\cO^\times_{\ol X_\bullet}\otimes\Ql\) has zero \(b=0\) row because the global units on every smooth
proper level are torsion, and its \(b=1\) row is \(\Pic(\ol X_a)\otimes\Ql=\NS(\ol X_a)\otimes\Ql\).
The bottom row begins with
\[
  0\to\rH^1_\Zar(\ol X,\cF)
  \xrightarrow{\;j_1\;}
  \bH^1_\Zar(\ol X_\bullet,\cH^1_\bullet(\Ql(1))).
\]
Thus \(d\log_\ell\) gives a commutative square
\[
\begin{tikzcd}[column sep=large]
\NS(\ol X)\otimes\Ql \arrow[r,"\pi_\bullet^*"] \arrow[d,"\delta"'] &
\NS(\ol X_\bullet)\otimes\Ql \arrow[d,"\mu"]\\
\rH^1_\Zar(\ol X,\cF) \arrow[r,"j_1"'] &
\bH^1_\Zar(\ol X_\bullet,\cH^1_\bullet(\Ql(1))),
\end{tikzcd}
\]
where \(\mu\) is the simplicial cycle class of Proposition~\ref{prop:mu}.

The equality \(\delta=u_*\circ\widetilde{c}_{1}\) follows from Kummer functoriality. Indeed
\[
  d\log_{\ell,\bullet}\circ\pi_\bullet^\#=
  \beta\circ\pi_\bullet^*d\log_{\ell,\ol X}
\]
as maps
\(\pi_\bullet^*(\cO_{\ol X}^{\times}\otimes\Ql)\to\cH^1_\bullet(\Ql(1))\). After applying
\(\pi_{\bullet,*}\), the left side is \(d\log_{\ell,\ol X}\) through
\(R^0\pi_{\bullet,*}(\cO^\times_{\ol X_\bullet}\otimes\Ql)=\cO^\times_{\ol X}\otimes\Ql\), while the
right side is \(u\circ d\log_{\ell,\ol X}\) by the definition of the augmentation \(u\). Passing to
\(\rH^1_\Zar\) gives the claimed factorisation.

Now \(\pi_\bullet^*\) is an isomorphism by Proposition~\ref{prop:NS-descent-Fpbar}, while
Proposition~\ref{prop:mu} identifies \(\operatorname{im}\mu\) with the weight-\(0\) subspace of the
simplicial group on the lower right. Since \(j_1\) is injective, the equality
\(j_1\delta=\mu\pi_\bullet^*\) makes \(\delta\) injective. If
\(x\in\rH^1_\Zar(\ol X,\cF)^{w=0}\), then \(j_1(x)\) is weight \(0\), hence
\(j_1(x)=\mu(y)\) for a unique \(y\in\NS(\ol X_\bullet)\otimes\Ql\). Write
\(y=\pi_\bullet^*(z)\) using N\'eron--Severi descent. Then
\(j_1(x)=j_1\delta(z)\), and the injectivity of \(j_1\) gives \(x=\delta(z)\). This proves the
weight-zero isomorphism.
\end{proof}

\begin{proof}[Proof of Theorem~\ref{thm:main-result}]
By Proposition~\ref{prop:reduction-to-seminormal}, it suffices to treat the case where \(X\) is also
seminormal. We maintain this assumption for the remainder of the proof, retaining
\(\pi_\bullet\colon X_\bullet\to X\) and \(\ol{X}_\bullet\) as fixed above.

By Proposition~\ref{prop:delta-cokernel}, \(\delta = u_*\circ\widetilde{c}_{1}\).
Frobenius-equivariance of \(\widetilde{c}_{1}\) (\ref{para:cycle-class}) and weight-\(0\) purity of
\(\NS(\ol{X})\otimes\Ql\) (\ref{para:neron-severi}) factor \(\widetilde{c}_{1}\) through
\(\rH^1_{\Zar}(\ol{X}, \cH^1)^{w=0}\) (Lemma~\ref{lem:Frobenius-module}(iii)). Restricting both sides to weight
zero,
\[
  \delta|_{w=0}\;=\;(u_*)^{w=0}\circ\widetilde{c}_{1}\colon\;\NS(\ol{X})\otimes\Ql\to\rH^1_\Zar(\ol{X},\cF)^{w=0}.
\]
The composite is an isomorphism (Proposition~\ref{prop:delta-cokernel}), and the second factor is a monomorphism
(Proposition~\ref{prop:u-inj}, which holds for any proper scheme over \(k\) and so applies in particular
to the seminormal \(X\)). A factorisation of an isomorphism through a monomorphism forces both factors to
be isomorphisms, so
\(\widetilde{c}_{1}\colon\NS(\ol{X})\otimes\Ql\xrightarrow{\sim}\rH^1_{\Zar}(\ol{X}, \cH^1)^{w=0}\) is the
asserted \(G_{k}\)-equivariant isomorphism in the seminormal case, and hence
(Proposition~\ref{prop:reduction-to-seminormal}) for all proper reduced \(k\)-schemes.
\end{proof}

\medskip\noindent\textbf{Relation to the Tate conjecture for divisors.}\par\smallskip

The \(\ell\)-adic first Chern class
\(c_{1}\colon\Pic(\ol{X})\otimes\Ql\to\rH^{2}_{\et}(\ol{X},\Ql(1))\) is the étale counterpart of
\(\widetilde{c}_{1}\). We compare it with the finite-order part of the full étale group.

\begin{para}[The finite-order part]\label{para:finite-order-part}
For a Frobenius module \((V,F)\) write
\(V_{\mathrm{fo}}:=\bigcup_{n\geq 1}\ker(F^{n}-\mathrm{id}_{V})\), the sum of the eigenspaces for
eigenvalues that are roots of unity. As a root of unity is a \(q\)-Weil number of weight \(0\), one has
\(V_{\mathrm{fo}}\subseteq V^{w=0}\) and \(V_{\mathrm{fo}}=(V^{w=0})_{\mathrm{fo}}\). A morphism \(f\colon V\to V'\) satisfies
\(f(V_{\mathrm{fo}})\subseteq V'_{\mathrm{fo}}\), and \(V\mapsto V_{\mathrm{fo}}\) is left exact, hence
commutes with equalizers. For \(\ol{X}\) the base change of a proper \(k\)-scheme,
\(\rH^{2}_{\et}(\ol{X},\Ql(1))_{\mathrm{fo}}\) consists of the classes that become Galois-invariant over a
finite extension of \(k\).
\end{para}

\begin{lem}\label{lem:fo-descent}
Let \(X\) be a proper reduced \(k\)-scheme and \(\pi_\bullet\colon X_\bullet\to X\) a smooth proper
hypercover, \(\ol{X}_\bullet=X_\bullet\times_k\ol{k}\). Cohomological descent (\ref{para:simp-setup})
restricts to a \(G_k\)-equivariant isomorphism
\[
  \rH^{2}_{\et}(\ol{X},\Ql(1))_{\mathrm{fo}}\;\xrightarrow{\;\sim\;}\;
  \operatorname{Eq}\Bigl(\rH^{2}_{\et}(\ol{X}_0,\Ql(1))_{\mathrm{fo}}\;\rightrightarrows\;\rH^{2}_{\et}(\ol{X}_1,\Ql(1))_{\mathrm{fo}}\Bigr)
\]
onto the equalizer of the two face-map pullbacks.
\end{lem}

\begin{proof}
The étale analogue of the skeletal spectral sequence \eqref{eq:skeletal-SS} for the simplicial étale
topos \((\ol{X}_\bullet)_{\et}\) of (\ref{para:simp-setup}),
\[
  E_{1}^{p,q}=\rH^{q}_{\et}(\ol{X}_p,\Ql(1))\;\Longrightarrow\;\bH^{p+q}_{\et}(\ol{X}_\bullet,\Ql(1)),
\]
is Frobenius-equivariant, and cohomological descent (\ref{para:simp-setup}) identifies its abutment with
\(\rH^{p+q}_{\et}(\ol{X},\Ql(1))\). Each \(\ol{X}_p\) is smooth proper over \(\ol{k}\), so \(E_{1}^{p,q}\)
is pure of weight \(q-2\) \cite[Th\'{e}or\`{e}me~1.6]{Del80}. By Lemma~\ref{lem:Frobenius-module}(ii) so
is every subquotient, and in total degree \(2\) the graded pieces \(E_{\infty}^{p,2-p}\) are pure of
weight \(-p\).

Write \(H=\rH^{2}_{\et}(\ol{X},\Ql(1))\) with abutment filtration \(F^{0}\supseteq F^{1}\supseteq
F^{2}\supseteq F^{3}=0\) and \(F^{p}/F^{p+1}=E_{\infty}^{p,2-p}\). Then \(F^{1}\) has weights
\(\subseteq\{-1,-2\}\). Applying the exact functor \((-)^{w=0}\)
(Lemma~\ref{lem:Frobenius-module}(iv)) to
\[
  0\to F^{1}\to H\to E_{\infty}^{0,2}\to 0
\]
gives
\(H^{w=0}\xrightarrow{\sim}E_{\infty}^{0,2}\). Since \(H_{\mathrm{fo}}=(H^{w=0})_{\mathrm{fo}}\)
(\ref{para:finite-order-part}), also \(H_{\mathrm{fo}}\xrightarrow{\sim}(E_{\infty}^{0,2})_{\mathrm{fo}}\).

At the edge \(p=0\) there are no incoming differentials, so
\(E_{\infty}^{0,2}=\bigcap_{r\geq2}\ker\bigl(d_{r}\colon E_{r}^{0,2}\to E_{r}^{r,3-r}\bigr)\subseteq
E_{2}^{0,2}\). For \(r\geq2\) the target is pure of weight \(1-r\leq-1\), so \(d_{r}\) vanishes on
weight-\(0\) parts (Lemma~\ref{lem:Frobenius-module}(iii)), and as \(V_{\mathrm{fo}}\subseteq V^{w=0}\)
we get \((E_{\infty}^{0,2})_{\mathrm{fo}}=(E_{2}^{0,2})_{\mathrm{fo}}\). Finally
\(E_{2}^{0,2}=\ker\bigl(d_{1}\colon E_{1}^{0,2}\to E_{1}^{1,2}\bigr)\) with \(d_{1}\) the difference of
the two face-map pullbacks. Left-exactness of \((-)_{\mathrm{fo}}\) (\ref{para:finite-order-part}) gives
\[
  (E_{2}^{0,2})_{\mathrm{fo}}=\operatorname{Eq}\bigl(\rH^{2}_{\et}(\ol{X}_0,\Ql(1))_{\mathrm{fo}}\rightrightarrows\rH^{2}_{\et}(\ol{X}_1,\Ql(1))_{\mathrm{fo}}\bigr).
\]
Composing the three isomorphisms proves the claim. Each is \(G_k\)-equivariant because the spectral
sequence and the descent isomorphism are \(G_k\)-equivariant (\ref{para:simp-setup}).
\end{proof}

\begin{prop}\label{prop:tate-divisors-fo}
Assume that for every smooth proper \(k\)-scheme \(Y\) the \(\ell\)-adic cycle class
\(\NS(\ol{Y})\otimes\Ql\to\rH^{2}_{\et}(\ol{Y},\Ql(1))_{\mathrm{fo}}\) is an isomorphism onto the
finite-order part (\ref{para:finite-order-part})\footnote{Equivalently, this is the Tate conjecture for
divisors \cite[Conjecture~$T^{1}$]{Tate94} over every finite extension of \(k\), since a finite-order class
becomes Frobenius-invariant over such an extension.}. Then for every proper reduced \(k\)-scheme \(X\) the
\(\ell\)-adic cycle class induces a \(G_k\)-equivariant isomorphism
\[
  c_{1}\colon\NS(\ol{X})\otimes\Ql\;\xrightarrow{\;\sim\;}\;\rH^{2}_{\et}(\ol{X},\Ql(1))_{\mathrm{fo}}.
\]
Combined with Theorem~\ref{thm:main-result}, the cycle class identifies all three groups:
\(\NS(\ol{X})\otimes\Ql\) maps isomorphically onto \(\rH^1_{\Zar}(\ol{X}, \cH^1)^{w=0}\) (by
\(\widetilde{c}_{1}\)) and onto \(\rH^{2}_{\et}(\ol{X},\Ql(1))_{\mathrm{fo}}\) (by \(c_{1}\)), so in
particular \(\rH^1_{\Zar}(\ol{X}, \cH^1)^{w=0}\cong\rH^{2}_{\et}(\ol{X},\Ql(1))_{\mathrm{fo}}\).
\end{prop}

\begin{proof}
Fix a smooth proper hypercover \(\pi_\bullet\colon X_\bullet\to X\)
(Proposition~\ref{prop:smooth-proper-sdh-hypercover}). For \(p=0,1\), \(\ol{X}_p\) is a finite disjoint
union of smooth proper varieties, so
\(\widetilde{c}_{1}\colon\NS(\ol{X}_p)\otimes\Ql\xrightarrow{\sim}\rH^1_{\Zar}(\ol{X}_p, \cH^1)\)
(Proposition~\ref{cor:cycle-class-smooth-proper}) followed by the Leray edge (Lemma~\ref{lem:leray-edge})
exhibits \(c_{1}\) on \(\ol{X}_p\) as injective with image in the finite-order part (\(\NS\otimes\Ql\) is
pure of weight \(0\) and \(F\)-finite by (\ref{para:neron-severi})). The Tate hypothesis makes it onto,
hence \(c_{1}\colon\NS(\ol{X}_p)\otimes\Ql\xrightarrow{\sim}\rH^{2}_{\et}(\ol{X}_p,\Ql(1))_{\mathrm{fo}}\).
These commute with the face-map pullbacks (functoriality of the cycle class, (\ref{para:cycle-class})),
hence identify equalizers. The source equalizer is \(\NS(\ol{X})\otimes\Ql\)
(Proposition~\ref{prop:NS-descent-Fpbar}) and the target is \(\rH^{2}_{\et}(\ol{X},\Ql(1))_{\mathrm{fo}}\)
(Lemma~\ref{lem:fo-descent}), so the induced map on equalizers is the desired \(G_k\)-equivariant
isomorphism \(c_{1}\) on \(\ol{X}\). The three-way identification combines this with
Theorem~\ref{thm:main-result}.
\end{proof}

\begin{example}[Sharpness of the weight-zero condition]\label{ex:normal-weight-example}
The weight-zero restriction in Theorem~\ref{thm:main-result} cannot be dropped, even for normal projective
surfaces. Let \(X\subset\bbP^{4}\) be the complete-intersection surface of Barbieri-Viale and Srinivas
\cite[Example~1]{BarbieriVialeSrinivas93},
\[
  X=\{x^{3}+y^{3}+z^{3}=0,\ \ ax^{2}+by^{2}+cz^{2}+dv^{2}=w^{2}\}\subset\bbP^{4},
\]
the double cover of the cubic cone over the elliptic curve \(E=\{x^{3}+y^{3}+z^{3}=0\}\), branched along a
general quadric section avoiding the vertex (\(\operatorname{char}k\neq2,3\), and the admissible \((a,b,c,d)\) form a nonempty Zariski-open locus, so such a section exists over \(\ol k\)). It is normal and projective with
two simple-elliptic singular points \(v_{0},v_{\infty}\). Its minimal resolution \(f\colon S\to X\) is the double
cover of the resolved cone (a \(\bbP^{1}\)-bundle over \(E\)) branched along the quadric section, with exceptional
curves \(C_{0},C_{\infty}\cong E\). Thus \(q(S)=1\) and \(\rH^{1}(\ol S,\Ql(1))=V:=\rH^{1}(\ol E,\Ql(1))\), each
restriction \(\rH^{1}(\ol S,\Ql(1))\to\rH^{1}(\ol C_{\iota},\Ql(1))\) an isomorphism
\cite[Example~1]{BarbieriVialeSrinivas93}. Here \(V\) is pure of weight \(-1\) and of dimension \(2\)
\cite[Th\'{e}or\`{e}me~1.6]{Del80}.

The abstract blow-up square with \(C_{0}\sqcup C_{\infty}\to S\) and \(\{v_{0},v_{\infty}\}\to X\) gives, by proper
base change \cite[Exp.~XII]{ArtinGrothendieckSGA4}, an exact boundary
\(\partial\colon\rH^{1}(\ol C_{0},\Ql(1))\oplus\rH^{1}(\ol C_{\infty},\Ql(1))\to\rH^{2}(\ol X,\Ql(1))\) whose
kernel is the diagonal image of \(\rH^{1}(\ol S,\Ql(1))\). Hence \(\partial\) embeds
\(\coker(V\xrightarrow{\Delta}V\oplus V)\cong V\) into \(\rH^{2}(\ol X,\Ql(1))\) as a weight-\((-1)\) subspace. It
is Zariski-locally trivial. On \(X\setminus\{v_{1-\iota}\}\) the local boundary vanishes, since the restriction
\(\rH^{1}(\ol S,\Ql(1))\to\rH^{1}(\ol C_{\iota},\Ql(1))\) is an isomorphism and factors through
\(\rH^{1}(\ol S\setminus\ol C_{1-\iota},\Ql(1))\), which therefore surjects onto \(\rH^{1}(\ol C_{\iota},\Ql(1))\). Hence the class
dies on the cover \(X=(X\setminus\{v_{\infty}\})\cup(X\setminus\{v_{0}\})\) and lies in
\(F^{1}\rH^{2}(\ol X,\Ql(1))=\rH^1_{\Zar}(\ol X, \cH^1)\) (Lemma~\ref{lem:leray-edge}). Therefore
\(\rH^1_{\Zar}(\ol X, \cH^1)\) contains the nonzero weight-\((-1)\) space \(V\) and is not pure of weight \(0\).
\end{example}

\begin{rem}
Over \(\bbC\), Barbieri-Viale and Srinivas \cite[Example~1]{BarbieriVialeSrinivas93} find that the discrepancy
between \(\NS(X)\) and \(\rH^{1}_{\Zar}(X,\cH^1)\) has rank \(7\). Besides \(V\) it carries a rank-\(5\) contribution from
the subgroup of \(\Pic^{0}(E)\) generated by the degree-zero parts of the six points
\(\{x^{3}+y^{3}+z^{3}=0\}\cap\{ax^{2}+by^{2}+cz^{2}=0\}\). Over \(\ol k\) that subgroup is torsion
(Lemma~\ref{lem:torsion-ag}), so it dies after \(\otimes\,\Ql\) and only the weight-\((-1)\) part \(V\) survives.
\end{rem}

\section*{Acknowledgements}

We thank Luca Barbieri-Viale, H\'{e}l\`{e}ne Esnault and Stefan Schreieder for helpful comments on an
earlier version of this article. KVS would like to thank the Department of Mathematics at the
University at Buffalo for its hospitality during a visit in which part of this work was carried out.

\bibliographystyle{alpha-custom}
\bibliography{bib/references}

\end{document}